\documentclass[reqno]{amsart}
\usepackage{paper-common}
\title{Quantitative Stability, Coercivity and Uniqueness of Optimal Transport Plans}
\author{William Ford$^1$}
\thanks{$^1$CMAP École Polytechnique, Palaiseau, France. Email: \texttt{william.ford@polytechnique.edu}}
\date{}
\hypersetup{
    hidelinks,
    hypertexnames=false,
    pdftitle={Quantitative Stability, Coercivity and Uniqueness of Optimal Transport Plans},
    pdfauthor={William Ford},
    pdfsubject={Quadratic optimal transport},
    pdfkeywords={optimal transport, quantitative stability, uniqueness, coercivity}
}

\begin{document}

\begin{abstract}
For probability measures  on $\R^d$ supported in fixed compact sets, we prove that quadratic optimal transport plans are quantitatively stable in Wasserstein distance under perturbation of both marginal measures, assuming that one of the initial measures satisfies an upper Ahlfors regularity condition with exponent strictly greater than $d-1$. Under the same assumption, we also prove novel quantitative stability results for optimal transport maps.

Furthermore, we prove a coercivity theorem which states that any probability measure on the product space must be quantitatively close to the set of optimal plans, if it has similar marginals and a similar transport cost to the optimal value.

Finally, we prove a quantitative uniqueness theorem which acts as a quantitative counterpart to Brenier's theorem. The Wasserstein diameter of the set of optimal plans is controlled by the distance of one marginal measure to a regular measure for which uniqueness holds. In this way, ``almost uniqueness'' of optimal plans is quantified by the source measure being ``almost regular''.

Examples are provided which prove that the exponents of all estimates are sharp.
    
    \vspace{0.2cm}
    
\end{abstract}

\maketitle
\vspace{-0.7cm}
\setcounter{tocdepth}{1}
\tableofcontents
\vspace{-0.9cm}

\section{Introduction}

\label{ch: introduction}
Given two probability measures $\rho$ and $\mu$ on $\R^d$, the quadratic optimal transport problem between $\rho$ and $\mu$ is the minimisation problem
\begin{equation}
\label{intro: quadratic OT problem}
    \min_{\gamma \in \Pi(\rho, \mu)} \int_{\R^d \times \R^d} \|x-y\|^2 \di \gamma(x,y),
\end{equation}
where $\Pi(\rho, \mu)$ denotes the set of all transport plans between $\rho$ and $\mu$, i.e. probability measures on $\R^d \times \R^d$ with first marginal $\rho$ and second marginal $\mu$. In this paper we study the following questions:
\begin{enumerate}[label=(\Roman*)]
    \item When are optimal plans quantitatively stable under perturbation of $\rho$ and $\mu$?
    \item When are plans with a nearly optimal cost quantitatively close to the set of optimisers?
    \item When is the set of  optimal plans quantitatively close to a singleton, i.e. when does it have small diameter?
\end{enumerate}
These questions are crucial to the many applications of optimal transport, for example in machine learning and computer vision \cite{courty2016optimal,makkuva2020optimal,arjovsky2017wasserstein}, statistics \cite{el2012bayesian,carlier2016vector,deb2023multivariate} and numerical analysis \cite{budd2015geometry,gallouet2018lagrangian}, where one would like to quantify the errors introduced by various approximations and estimations. To respond quantitatively to the above questions we require a distance in which we will measure the perturbations of the marginals, as well as another to compare transport plans. For $r \in [1, \infty]$, the $r$-Wasserstein distance between $\rho$ and $\mu$ is given by
\begin{equation}
    W_r(\rho, \mu) = \min_{\gamma \in \Pi(\rho, \mu)} \| x-y\|_{L^r(\gamma)},
\end{equation}
with problem \eqref{intro: quadratic OT problem} corresponding to the case $r=2$. We will measure perturbations in the marginal measures using Wasserstein distances on $\R^d$. Solutions to \eqref{intro: quadratic OT problem} are also probability measures, and will be compared instead using Wasserstein distances on $\R^{2d}$.

\subsection{Contributions}

For a subset $E$ of Euclidean space, denote by $\PP(E)$ the set of probability measures on $E$. Throughout this paper we assume $d\geq 2$. For marginals $\rho$ and $\mu$, we denote by $\Gamma(\rho, \mu)$ the set of all solutions to problem \eqref{intro: quadratic OT problem}, and $\gamma_{\rho \to \mu}$ will denote any optimiser to \eqref{intro: quadratic OT problem}. 

The four main results of this paper are stated below: stability of optimal plans under perturbation of both marginals and optimal maps with fixed source, coercivity of the objective with respect to transport plans on $\R^{2d}$ with arbitrary marginals, and quantitative uniqueness of solutions near a regular source measure. All the theorems are proven in the text for $W_q$ distances between plans and $L^q$ between maps, for any $q \geq1$. In the introduction we present only the case $q=2$ for readability. The principal assumption in all theorems is that one of the measures satisfies the following regularity condition.
\begin{assumption}
\label{ass: alfors}
    Assume $\rho \in \PP_2(\R^d)$ such that for some $M>0$ and $\theta \in (0, 1]$,
    \begin{equation}
    \rho(B_\eta(x)) \leq M \eta^{d-1 + \theta} \quad\text{ for every } x \in \R^d \text{ and } \eta>0.
    \tag*{$A(M,\theta)$}
    \label{eq: assumption eqn}
    \end{equation}
\end{assumption}
A foundational result in optimal transport due to Brenier and McCann \cite{brenier1991polar, mccann1995existence} states that if $\rho$ vanishes on all sets of Hausdorff dimension $d-1$, then solutions to \eqref{intro: quadratic OT problem} are unique for any $\mu$, and the optimal plan is induced by a transport map, which we denote $T_{\rho \to \mu}$. Our assumption \ref{eq: assumption eqn} is strictly stronger than this, and can be viewed as a quantitative form of the Brenier-McCann hypothesis. 
\begin{theorem*}[Bi-marginal plan stability: Theorem~\ref{thrm: bi marginal plan stability}, Proposition~\ref{prop: sharp exponent stability}]
    Let $\X$ and $\Y$ be compact subsets of $\R^d$ for $d\geq2$, and let $\rho_0$ be a probability measure on $\X$ satisfying \ref{eq: assumption eqn}. Then for all $\rho_1 \in \PP(\X)$, all $\mu_0, \mu_1 \in \PP(\Y)$, and any $\gamma_{\rho_1 \to \mu_1} \in \Gamma(\rho_1, \mu_1)$,
    \begin{equation}
        W_2(\gamma_{\rho_0 \to \mu_0}, \gamma_{\rho_1 \to \mu_1}) \leq C\Big(W_1(\rho_0, \rho_1) + W_1(\mu_0, \mu_1)\Big)^{\frac{\theta}{2(1+\theta)}},
    \end{equation}
    for an explicit constant $C(\X, \Y, M)>0$ given in the statement of Theorem \ref{thrm: bi marginal plan stability}. Moreover, the bound is sharp in exponent for all such $\theta$.
\end{theorem*}
Here, by sharpness, we mean that in general, the same inequality cannot hold with a larger exponent. The example given in Proposition \ref{prop: sharp exponent stability} is for fixed source, perturbing only the target. Thus we do not rule out the possibility of a better exponent under perturbations of the source, with fixed target.

Our second stability theorem is for optimal maps. As soon as both marginals are perturbed, a Brenier map may not exist for the new problem. We can nonetheless discuss bi-marginal stability of the gradient of the Brenier potential, which has implications for statistical estimation of transport maps; see Remark \ref{rmk: statistical estimation of transport maps}. For clarity, we state here the fixed-source version where maps remain well defined, leaving the general case to Theorem~\ref{thrm: bi marginal map stability}.

\begin{theorem*}[Fixed source map stability: Theorem~\ref{thrm: bi marginal map stability}, Proposition~\ref{prop: sharp exponent stability}]
    Let $\X$ and  $\Y$ be compact subsets of $\R^d$ for $d \geq2$, and let $\rho$ be a probability measure on $\X$ satisfying \ref{eq: assumption eqn}. Then
    \begin{equation}
        \forall \mu_0, \mu_1 \in \PP(\Y), \quad \|T_{\rho \to \mu_0} - T_{\rho \to \mu_1}\|_{L^2(\rho)} \leq C W_1(\mu_0, \mu_1)^{\frac{\theta}{2(1+ \theta)}},
    \end{equation}
    for an explicit constant $C(\X, \Y, M)>0$ given in the statement of Theorem \ref{thrm: bi marginal map stability}. Moreover, the bound is sharp in exponent for all such $\theta$.
\end{theorem*}

Observe that for fixed source measure $\rho$, and  plans $\gamma_{\rho \to \mu_i} = (\id, T_{\rho \to \mu_i})_\#\rho$ induced by maps $T_{\rho \to \mu_i}$, $i=0, 1$, the coupling $(\id, T_{\rho \to \mu_0}, \id, T_{\rho \to \mu_1})_\#\rho \in \Pi(\gamma_0, \gamma_1)$ provides the upper bound
\begin{equation*}
    W_q(\gamma_{\rho \to \mu_0}, \gamma_{\rho \to \mu_1}) \leq \|T_{\rho \to \mu_0} - T_{\rho \to \mu_1}\|_{L^q(\rho)}
\end{equation*}
so that in this case, map stability implies plan stability with fixed source, for the same constant and exponent.

Our third theorem gives a quantitative coercivity estimate for the transport problem. It says that for the transport problem between $\rho_0$ satisfying \ref{eq: assumption eqn} and any $\mu_0$, any transport plan $\gamma_1 \in \PP(\X \times \Y)$ (not necessarily optimal between its marginals) is quantitatively close to the optimal plan $\gamma_{\rho_0 \to \mu_0}$, under the assumption that $\gamma_1$ has similar marginals to $\rho_0$ and $\mu_0$ and realises a similar transport cost $\int \|x-y\|^2 \di \gamma_1$ to $\gamma_{\rho_0 \to \mu_0}$.

\begin{samepage}
\begin{theorem*}[Coercivity: Theorem~\ref{thrm: coercivity general marginals}, Proposition~\ref{prop: sharp exponent coercivity}]
    Let $\X$ and $\Y$ be compact subsets of $\R^d$ for $d \geq 2$, let $\rho_0$ be a probability measure on $\X$ satisfying \ref{eq: assumption eqn}, and fix any $\mu_0 \in \PP(\Y)$. Then for any $\gamma_1 \in \PP(\X \times \Y)$ with marginals $\rho_1 \in \PP(\X)$ and $\mu_1 \in \PP(\Y)$,
    \begin{equation}
         \quad W_2(\gamma_1, \gamma_{\rho_0 \to \mu_0}) \leq C \left( \left|\int \|x-y\|^2 \di \gamma_1 - W_2^2(\rho_0, \mu_0)\right| + W_1(\rho_0, \rho_1) + W_1(\mu_0, \mu_1)\right)^\frac{\theta}{2(1+\theta)}
    \end{equation}
    for an explicit constant $C(\X, \Y, M) >0$ given in the statement of Theorem \ref{thrm: coercivity general marginals}. Moreover, the bound is sharp in exponent for all such $\theta$.
\end{theorem*}
\end{samepage}
The example given in Proposition \ref{prop: sharp exponent coercivity} is for fixed marginals, so that the exponent is still sharp even restricted to this case. In the text we state and prove two other coercivity theorems. Theorem \ref{thrm: coercivity rough measures} concerns transport problems where one marginal is close in Wasserstein distance to some regular measure, and coercivity holds up to an error quantified in terms of this distance. The second, Theorem \ref{thrm: map coercivity wrt dual suboptimality}, concerns the maps constructed out of the subgradients of nearly optimal potentials for the dual problem.

Our final theorem  gives a bound on the diameter of the set of optimal transport plans in terms of the distance of one marginal measure to a measure satisfying \ref{eq: assumption eqn}. This provides a quantification of the Brenier-McCann theorem, saying that if one marginal is close to regular then the set of optimisers is close to being a singleton, regardless of the target measure.  For $q \in [1,+\infty)$ and probability measures $\rho$ and $\mu$, denote the $W_q$ diameter of the optimal set by
\begin{equation*}
    \diam_{W_q}\Gamma(\rho,\mu)
:= \sup_{\gamma_0,\gamma_1 \in \Gamma(\rho,\mu)} W_q(\gamma_0,\gamma_1).
\end{equation*}
\begin{samepage}
\begin{theorem*}[Quantitative uniqueness: Theorem~\ref{thrm: quant unique general}, Proposition~\ref{prop: sharpness quant unique}]
Let $\X$ and $\Y$ be compact and convex subsets of $\R^d$ for $d\geq 2$, and let $\tilde\rho$ be a probability measure on $\X$ satisfying \ref{eq: assumption eqn}. Then for all $\rho\in\PP(\X)$ and $\mu\in\PP(\Y)$,
\begin{equation}
\label{eq: quant unique intro}
        \diam_{W_2}\Gamma(\rho,\mu)
        \leq C W_2(\rho,\tilde\rho)^{\frac{\theta}{2+\theta}},
\end{equation}
for an explicit constant $C(\X, \Y, M)>0$ given in the statement of Theorem \ref{thrm: quant unique general}. Moreover, the bound is sharp in exponent for all such $\theta$.
\end{theorem*}
\end{samepage}

\subsection{Existing results}

Simple examples in \cite{ford2025quantitative} involving discrete measures show that without any regularity assumptions, problem \eqref{intro: quadratic OT problem} is not quantitatively stable. OT stability literature can be divided into two categories of assumptions: local stability in the neighbourhood of a transport problem with a regular map/potential, and local stability in the neighbourhood of a transport problem with at least one regular marginal measure. All stability results in this paper are of the second type.

For the first type of result, in \cite{gigli2011holder} it is proved that with fixed probability density $\rho$ and some target $\mu$, if $T_{\rho \to \mu}$ is Lipschitz, then OT maps are locally quantitatively stable in $L^2(\rho)$ with respect to $W_1$ perturbations of the target $\mu$. This is generalised to bi-marginal plan stability under the same assumption in \cite{li2021quantitative}. Assuming the transport from $\rho_0$ to $\mu_0$ has a class $\mathcal{C}^{1, 1}$ Brenier potential (see \eqref{eq: Brenier potential definition} below for the definition), Li and Nochetto establish the control
\begin{equation*}
\forall \rho_1 \in \PP(\X) \forall \mu_1 \in \PP(\Y), \quad W_2(\gamma_{\rho_0 \to \mu_0}, \gamma_{\rho_1 \to \mu_1}) \lesssim \Big(W_2(\rho_0, \rho_1) + W_2(\mu_0, \mu_1)\Big)^\frac{1}{2},
\end{equation*}
for a constant depending on $\X$, $\Y$, and the Lipschitz constant of $\nabla \phi_{\rho_0 \to \mu_0}$. We also mention \cite{gallouet2025strong} which adapts this assumption to a larger class of transport costs, as well as \cite{manole2024plugin} which derives statistical results under related regularity assumptions. Unfortunately, in applications, such regularity assumptions can be difficult to verify and need not hold. OT maps are often not even continuous, let alone Lipschitz, even between regular densities. 

Assuming instead that one marginal measure is regular is often more relevant to applications. Much progress has been made recently in this direction, in particular for stability of OT maps with a regular fixed source $\rho$, under perturbation of the target $\mu$. We mention \cite{berman2021convergence} for the first such result, followed by \cite{delalande2023quantitative, letrouit2024gluing} in which a fixed-source stability theory is developed directly for Brenier potentials -- and consequently maps -- using tools from functional inequalities. For $\X \subset \R^d$ a bounded Lipschitz domain and $\rho$ a probability density on $\X$ bounded above and below by positive constants, they prove that for all $\mu_0, \mu_1 \in \PP(\Y)$,
\begin{equation*}
     \|\phi_{\rho \to \mu_0} -\phi_{\rho \to \mu_1}\|_{L^2(\rho)} \lesssim W_1(\mu_0, \mu_1)^\frac{1}{2}\quad\; \text{ and } \quad\; \|T_{\rho \to \mu_0} - T_{\rho \to \mu_1}\|_{L^2(\rho)} \lesssim W_1(\mu_0, \mu_1)^\frac{1}{6},
\end{equation*}
where the potentials are normalised by $\int_\X\phi_{\rho\to\mu_i}\di\rho=0$, and the constants depend on $\X$, $\Y$ and the upper and lower density bounds for $\rho$. We refer the reader to the lecture notes \cite{letrouit2025lectures} for more details on these methods. The current state of the art papers \cite{merigot2026sharp, cazelles2026statistical} prove OT map and plan stability for a larger class of source measures. Notably, the support of the regular measure need not be connected and the density can vanish arbitrarily fast. In \cite{merigot2026sharp}, Mérigot proves that for $\rho$ a bounded, compactly supported probability density on $\R^d$,
\begin{equation*}
  \forall \mu_0, \mu_1 \in \PP(\Y), \quad \|T_{\rho \to \mu_0} - T_{\rho \to \mu_1}\|_{L^2(\rho)} \lesssim W_1(\mu_0, \mu_1)^\frac{1}{4},
\end{equation*} 
for a constant depending on $\Y$ and $\rho$. Due to the example in \cite{letrouit2026unstable}, the exponent $1/4$ is sharp, matching the special case of our Theorem~\ref{thrm: bi marginal map stability} and Proposition~\ref{prop: sharp exponent stability} when $\theta = 1$ and $q=2$. In \cite{cazelles2026statistical}, Cazelles, Pauwels and Portales prove that for $\rho_0$ a bounded probability density on $\X$, for all $\rho_1 \in \PP(\X)$ and all $\mu_0, \mu_1 \in \PP(\Y)$,
\begin{equation*}
W_2(\gamma_{\rho_0 \to \mu_0}, \gamma_{\rho_1 \to \mu_1}) \lesssim \Big( W_2(\rho_0, \rho_1) + W_2(\mu_0, \mu_1)\Big)^{\frac{1}{8}},
\end{equation*}
for a constant depending on $\diam(\X)$, $\diam(\Y)$ and $\rho_0$.
Here their exponent is not sharp. Using our Theorem~\ref{thrm: bi marginal plan stability} with $\theta = 1$ and $q=2$ improves their exponent from $1/8$ to $1/4$, since by Jensen's inequality $W_1 \leq W_2$. Our result is still not equal to the theoretical upper bound of $1/3$ given by Proposition~\ref{prop: sharp exponent stability}. The above results are consequences of the paper \cite{carlier2024pushforward} on the quantitative regularity of convex functions. Our techniques in this paper extend the quantitative regularity theory to allow for our new parameters $\theta$ and $q$.

Coercivity of problem \eqref{intro: quadratic OT problem} has also received some study. In \cite{li2021quantitative}, Li and Nochetto prove that any transport problem satisfying the same $\mathcal{C}^{1, 1}$ regularity assumption on a Brenier potential is coercive in a similar sense to our Theorem~\ref{thrm: coercivity general marginals} presented above.
In \cite{merigot2026sharp}, Mérigot proves two coercivity results for plans and maps with respect to both dual and primal suboptimality with fixed marginals. These form the basis for our Theorem \ref{thrm: map coercivity wrt dual suboptimality} and Theorem \ref{thrm: coercivity rough measures}. We also mention \cite{carlier2023fenchel} which proves coercivity of the distance of a plan to the subgradient of the Brenier potential with respect to the primal suboptimality gap, as well as \cite{carlier2024barycentres,ford2025quantitative} which prove a strong convexity inequality for the dual problem which acts as a dual coercivity result for Brenier potentials.

The current literature on uniqueness of the primal problem is, to the best of our knowledge, only qualitative. The foundational results of Brenier and McCann \cite{brenier1991polar, mccann1995existence} state that if $\rho$ vanishes on all sets of Hausdorff dimension $d-1$, then solutions to \eqref{intro: quadratic OT problem} are unique for any $\mu$, and the optimal plan is induced by a transport map. Our assumption \ref{eq: assumption eqn} can be in some sense seen as a quantitative strengthening of this Brenier-McCann hypothesis. The sharp hypotheses for qualitative uniqueness have since been identified as slightly weaker. To guarantee that the optimal transport is unique and map induced for any $\mu$, it is necessary and sufficient that $\rho$ vanishes on any hypersurface which can be locally written as the graph of a difference of convex functions; see \cite{gigli2011inverse,zajivcek1979differentiation}. For the dual OT problem, quantitative uniqueness has been recently investigated by the author, in \cite{ford2026quantitative}. Here almost connectedness of one marginal measure's support implies a quantitative bound on the diameter of the set of optimal dual potentials.

\subsection{Idea of the proofs}

Since
\begin{equation*}
    \frac{1}{2}W_2^2(\rho, \mu) = \frac{1}{2}\int \|x\|^2 \di \rho(x) + \frac{1}{2}\int\|y\|^2\di\mu(y) - \sup_{\gamma \in \Pi(\rho, \mu)} \int \langle x, y \rangle \di \gamma(x, y),
\end{equation*}
problem \eqref{intro: quadratic OT problem} is equivalent to the maximum correlation problem. Its dual formulation is
\begin{equation}
    \label{eq: Brenier potential definition}
    \sup_{\gamma \in \Pi(\rho, \mu)} \int \langle x, y \rangle \di \gamma = \inf_{\phi \text{ convex}} \int \phi(x) \di \rho(x) + \int \phi^*(y) \di \mu(y),
\end{equation}
where $\phi^*(y) =\sup_{x \in \R^d} \langle x, y \rangle - \phi(x)$ denotes the Legendre transform. We call any minimiser $\phi$ a Brenier potential. Primal and dual optimality imply that for $\phi$ a Brenier potential and $\gamma$ an optimal plan,
\begin{equation*}
    \spt\gamma \subset\graph(\partial\phi).
\end{equation*}
The main tool used for all the results in this paper is Theorem~\ref{thrm: quant reg}, which proves that for $q \geq 1$, a measure $\tilde \rho$ satisfying \ref{eq: assumption eqn}, and a convex Lipschitz function $\phi : \R^d \to \R$, 
\begin{equation}
    \label{eq: quant reg explainer simplified}
    \forall \eta\geq 0, \quad \int_\X \diam \partial\phi(\overline B_\eta(x))^q \di \tilde \rho(x) \lesssim \eta^\theta,
\end{equation}
where here
\begin{equation*}
    \partial\phi(\overline B_\eta(x)) = \bigcup_{z \in \overline B_\eta(x)}\partial \phi(z).
\end{equation*}
This extends the result of \cite{carlier2024pushforward}, which proves \eqref{eq: quant reg explainer simplified} when $q>1$ and $\rho$ is the Lebesgue measure and consequently for all bounded densities, i.e. $\rho$ satisfying \ref{eq: assumption eqn} with $\theta = 1$. This result should be seen as a ``quantitative regularity'' result for convex functions, with the integral quantity representing average local oscillation at scale $\eta$.

Once we have \eqref{eq: quant reg explainer simplified}, then the proof of Theorem~\ref{thrm: quant unique general}, quantitative uniqueness, is direct. We sketch the argument for the $W_1$ plan - $W_1$ marginal bound here. Fix any two optimal plans $\gamma_0, \gamma_1 \in \Gamma(\rho, \mu)$ and a Brenier potential $\phi$, and let $\alpha \in \Pi(\tilde \rho,\rho)$ be an optimal coupling for $W_1(\tilde \rho,\rho)$. Then, for any $\eta >0$, by choosing a coupling $\pi \in \Pi(\gamma_0, \gamma_1)$ which couples mass according to the identity for the first marginal $\rho$ of each plan (see \eqref{eq: pi beta defn} with $\beta \in \Pi(\rho, \rho)$ concentrated along the diagonal),
\begin{align*}
    W_1(\gamma_0, \gamma_1) \leq& \int_\X \diam \partial\phi(x) \di \rho(x) = \int_{\X^2} \diam \partial\phi(x) \di \alpha(\tilde x, x)\\
    \leq& \int_{\{\|\tilde x - x\| \leq \eta\}} \diam \partial\phi(\overline B_\eta(\tilde x)) \di \alpha(\tilde x, x)
    + D_\Y\alpha(\|x-\tilde x\|>\eta)\\
    \lesssim& \eta^\theta + \frac{W_1(\tilde \rho, \rho)}{\eta}.
\end{align*}
The bound follows by choosing $\eta = W_1(\tilde \rho, \rho)^{1/(1+\theta)}$.

The stability and coercivity theorems are less direct, as we must compare two plans which do not have the same marginals, and are not, in general, concentrated on the graph of the subgradient of the same convex function. For these bounds, we also couple transport plans in a manner which is optimal for the first marginal. The main term we are left to bound in this case is a quantitative measure of the suboptimality of one plan to being in the subgradient of the other problem's Brenier potential, of the form
\begin{equation*}
    \int_{\X \times \Y}G_{\phi_0}(x, y) \di \gamma_1(x, y),
\end{equation*}
where $G_\phi(x,y) = \phi(x) + \phi^*(y) - \langle x, y \rangle \geq 0$ is the slackness in the Fenchel--Young inequality at $(x, y)$ for a convex function $\phi$. In the coercivity theorem, this quantity is bounded by the difference in transport costs between $
\gamma_0$ and $\gamma_1$, plus the sum of $W_1$ differences between their marginals. In the stability theorem, this quantity is bounded directly by the $W_1$ differences between their marginals.

\section{Quantitative regularity: coupling transport plans}
\label{ch: quantitative regularity}
In the majority of settings considered in this paper, objects are invariant under translations of $\X$ and $\Y$, so we may assume without loss of generality that $0$ belongs to both sets. It will also cost us no loss of generality to assume they are convex since we never ask measures to have full support. We denote the radii and diameters of each set by
\begin{equation*}
    R_\X = \sup_{x \in \X} \|x\|, \quad R_\Y = \sup_{y \in \Y} \|y\|, \quad D_\X = \diam(\X), \quad D_\Y = \diam(\Y).
\end{equation*}
When $0 \in \X \cap \Y$, we have $R_\X \leq D_\X$ and $R_\Y \leq D_\Y$. In general, we will always assume that a pair of Brenier and dual $(\phi, \psi)$ potentials is of the form
    \begin{equation}
    \label{eq: double legendre trans eq}
        \phi(x) = \sup_{y \in \Y} \langle x, y \rangle  - \psi(y) \quad \text{ and } \quad \psi(y) = \sup_{x \in \X} \langle x, y \rangle - \phi(x),
    \end{equation}
rather than taking the Legendre transform with respect to all $\R^d$. Every optimal potential admits $\rho$ and $\mu$ a.e. representatives of this form; see \cite[Appendix B]{ford2026quantitative} for a rigorous justification of this assumption. Consequently, $\phi$ and $\psi$ are $R_\Y$ and $R_\X$ Lipschitz respectively, as envelopes of uniformly Lipschitz functions.

\subsection{An integral bound}
For a convex function $\phi:\R^d\to\R$ and $\eta\geq0$, denote
\begin{equation}
    D_\eta \phi(x):=\diam\partial\phi(\overline B_\eta(x)).
\end{equation}
The following theorem generalises \cite[Cor. 2.2]{carlier2024pushforward} to allow the exponent $q=1$ and measures satisfying \ref{eq: assumption eqn}. These bounds build on the previous works \cite{alberti1992singularities,alberti1994structure}, which study the size of singular sets of convex functions, rather than the average local oscillation quantity considered here. We first prove a lemma which gives an interior $L^\infty$ estimate for gradients of convex functions, as a consequence of monotonicity. Similar estimates exist in the literature, for example \cite{gutierrez2022estimates, bouchitte2007new}.
\begin{lemma}
\label{lem: diam to avg int lemma}
    Let $\phi : \R^d \to \R$ be a convex function. Then for any $x,c \in \R^d$ and any $\eta>0$,
    \begin{equation*}
        D_\eta \phi(x) \leq \frac{2(d+1)}{\beta_{d-1}\eta^d} \int_{B_{2\eta}(x)} \| \nabla \phi(z) - c\| \di z,
    \end{equation*}
    where $\beta_d := |B_1|$ is the volume of the $d$-dimensional unit ball. 
\end{lemma}
\begin{proof}
    Subtracting $\langle \cdot, c\rangle$ from $\phi$, it suffices to consider $c=0$. Choose $z \in \overline B_\eta(x)$ and $y \in \partial\phi(z)$. For all $z' \in B_\eta(z)$ and $y' \in \partial\phi(z')$, monotonicity implies
    \begin{equation*}
        \langle z' - z, y \rangle \leq \langle z'-z, y'\rangle \leq \eta \|y'\|.
    \end{equation*}
    Integrating in $z'$ over the half ball $A = \{ z' \in B_\eta(z) : \langle z' - z, y \rangle \geq 0 \} \subset B_{2\eta}(x)$ gives
    \begin{equation}
        \frac{\beta_{d-1}}{d+1} \eta^{d+1} \|y\| \leq \eta \int_{B_{2\eta}(x)} \| \nabla \phi(z')\| \di z'.
    \end{equation}
    Thus, for every $z \in \overline B_\eta(x)$ and $y \in \partial\phi(z)$,
    \begin{equation}
    \label{eq: L infinity control of convex functions}
        \|y\| \leq \frac{d+1}{\beta_{d-1} \eta^d} \int_{B_{2\eta}(x)} \| \nabla \phi(z')\| \di z'.
    \end{equation}
    Passing to the supremum over such $y$ gives the result. 
\end{proof}

\begin{theorem}
\label{thrm: quant reg}
Let $d \geq 2$, and let $\phi : \R^d \to \R$ be convex, with $\partial\phi(\R^d) \subset \Y$ for some compact convex $\Y\subset \R^d$. Let $R>0$, and $\tilde \rho \in \PP(\overline B_R)$ satisfy \ref{eq: assumption eqn}. Then, for every $q \in [1,\infty)$,
\begin{equation}
\label{eq: Lq quant reg}
    \forall \eta \geq 0, \qquad \int_{\overline B_R} D_\eta \phi(x)^q\di \tilde \rho(x)
    \leq c_d M R^{d-1+\theta} D_\Y^q\left(\frac{\eta}{R}\right)^{\theta},
\end{equation}
where
\begin{equation*}
    c_d=\frac{4}{3}\frac{6^d d^2(d+1)\beta_d}{\beta_{d-1}}.
\end{equation*}
\end{theorem}
\begin{proof}
If $D_\Y=0$, then $\phi$ is affine and the result is immediate. Otherwise, fix $x \in \R^d$ and $\eta>0$, and set
\begin{equation*}
    c_x = \frac{1}{|B_{2\eta}(x)|}\int_{B_{2\eta}(x)} \nabla \phi(z) \di z.
\end{equation*}
Applying Lemma~\ref{lem: diam to avg int lemma},
\begin{equation*}
    D_\eta \phi(x) \leq \frac{2(d+1)}{\beta_{d-1}\eta^d} \|\nabla\phi - c_x\|_{L^1(B_{2\eta}(x))}.
\end{equation*}
The $BV$ Poincaré inequality \cite[Theorem 3.44]{ambrosio2000functions}, applied to each component with the diameter bound and sharp universal coefficient $1/2$ \cite{acosta2004optimal}, gives
\begin{equation*}
    \|\nabla\phi - c_x\|_{L^1(B_{2\eta}(x))} \leq 2 \eta \|D^2\phi\|_{1, 1}(B_{2\eta}(x)),
\end{equation*}
where $D^2\phi$ (and later $\Delta \phi$) is a Radon measure and
\begin{equation*}
    A \in \R^{d \times d}, \quad\|A\|_{1, 1} = \sum_{i=1}^d \sum_{j=1}^d |A_{ij}|
\end{equation*}
denotes the entry-wise $l^1$ norm of a matrix. For a symmetric positive semi-definite matrix $A \in \R^{d \times d}$, by the inequality $|A_{ij}| \leq \frac{A_{ii}+A_{jj}}{2}$, it holds that $\|D^2\phi\|_{1,1} \leq d \Delta \phi$ in the sense of measures. Thus
\begin{equation}
\label{eq: local Hessian measure bound}
    D_\eta \phi(x) \leq \frac{4(d+1)}{\beta_{d-1} \eta^{d-1}}
     \|D^2\phi\|_{1, 1}(B_{2\eta}(x)) \leq \frac{4d(d+1)}{\beta_{d-1}\eta^{d-1}} \Delta\phi(B_{2\eta}(x)).
\end{equation}
Integrating \eqref{eq: local Hessian measure bound} in $x$ and using Fubini's theorem gives
\begin{align*}
    \int_{\overline B_{R}}D_\eta \phi(x)\di \tilde \rho(x) &\leq \frac{4d(d+1)}{\beta_{d-1}\eta^{d-1}} \int_{\overline B_{R}} \int_{B_{2\eta}(x)} \di \Delta\phi(z) \di \tilde \rho(x)\\
    &= \frac{4d(d+1)}{\beta_{d-1}\eta^{d-1}}\int_{B_{R+2\eta}} \tilde \rho(\overline B_{R} \cap B_{2\eta}(z)) \di \Delta \phi(z)\\
    &\leq \frac{2^{d+2}d(d+1)}{\beta_{d-1}} M\Delta \phi(B_{R+2\eta})\eta^\theta.
\end{align*}
Let $\chi \in \mathcal{C}^\infty_c(B_1)$ be a non-negative mollifier with $\int \chi = 1$, set $\chi_\varepsilon(x) = \varepsilon^{-d} \chi(x/\varepsilon)$, and define $\phi_\varepsilon = \phi * \chi_\varepsilon$. If $\phi$ is convex and $L$-Lipschitz, integration by parts yields
\begin{equation*}
    \Delta \phi_\varepsilon(B_{R+2\eta})\leq \lip(\phi_\varepsilon)\mathcal{H}^{d-1}(\partial B_{R+2\eta})
    \leq d\beta_d L(R+2\eta)^{d-1}.
\end{equation*}
Passing to the limit $\varepsilon \to 0$,
\begin{equation*}
    \Delta\phi(B_{R+ 2\eta}) \leq \liminf_{\varepsilon \to 0} \Delta \phi_\varepsilon(B_{R+2\eta}) \leq d\beta_d L(R+2\eta)^{d-1},
\end{equation*}
by the lower-semicontinuity of the measure of an open set for the convergence $\Delta \phi_\varepsilon \overset{\ast}{\rightharpoonup} \Delta \phi$. Noting that $\Delta\phi(A) = \Delta(\phi + \langle\cdot, c\rangle )(A)$ for any constant $c \in \R^d$, $L$ can be replaced by $D_\Y$. By assumption, for any $\eta>0$, $D_\eta \phi\leq D_\Y$, and hence
\begin{equation*}
    D_\eta \phi(x)^q\leq D_\Y^{q-1}D_\eta  \phi(x),
\end{equation*}
which proves that for all $\eta> 0$,
\begin{equation}
    \label{eq: local integral regularity}
    \int_{\overline B_{R}} D_\eta \phi(x)^q\di \tilde \rho(x)
    \leq 3^{1-d}c_d M (R+2 \eta)^{d-1} D_\Y^q\eta^{\theta}.
\end{equation}
If $\eta \leq R$, the stated bound is proven using $(R+ 2\eta) \leq 3R$; otherwise if $\eta \geq R$ the bound holds trivially since $1= \tilde \rho(\overline B_R) \leq M R^{d-1+ \theta}$ and $D_\eta \phi \leq D_\Y$ for all $\eta$. Finally, letting $\eta\downarrow0$ proves the case $\eta=0$.
\end{proof}

\begin{remark}[Sharpness of the exponent]
    The exponent of $\eta$ is sharp in general, as the following example demonstrates. Let $\phi(x) = |x_1|$ and $\rho(x) = c_\theta|x_1|^{\theta-1} \chi_{[-1, 1]^d}(x)$ for $c_\theta = \theta 2^{-d}$ a normalising constant. Then
    \begin{equation*}
        D_\eta \phi(x) = \begin{cases}
            2 & |x_1| \leq \eta\\
            0 & |x_1| > \eta,
        \end{cases}
    \end{equation*}
    so that for small $\eta$,
    \begin{equation*}
        \int_{[-1, 1]^d} D_\eta \phi (x)^q \di \rho(x) = 2^q \rho(|x_1| \leq \eta) = 2^q \eta^\theta.
    \end{equation*}
\end{remark}

\subsection{The Fenchel gap}

For $(x, y) \in \R^d \times \R^d$ and a convex function $\phi : \R^d \to \R$,  denote the remainder in the Fenchel-Young inequality by
\begin{equation*}
    G_\phi(x, y) = \phi(x) + \phi^*(y) - \langle x, y \rangle \geq 0,
\end{equation*}
where here $\phi^*(y)=\sup_{z\in\R^d}\langle z,y\rangle-\phi(z)$ is the convex conjugate. Equality $G_\phi(x, y) = 0$ is equivalent to the pair $(x, y)$ belonging to the graph of the subgradient of $\phi$, i.e. $y \in \partial \phi(x)$. In \cite[Lemma 1.1]{carlier2023fenchel}, Carlier proves a quantitative bound
\begin{equation*}
    G_\phi(x, y) \geq \frac{1}{2} \dist((x, y), \graph(\partial \phi))^2,
\end{equation*}
i.e. points which almost attain equality in Fenchel-Young are quantitatively close to  the set $\graph(\partial\phi)$.
The following lemma controls the distance of a specific $y$ to elements of $\partial \phi(x)$, rather than a pair $(x, y)$ to $\graph(\partial \phi)$.
\begin{lemma}
\label{lem: Fenchel bound distance}
    Let $\phi:\R^d \to \R$ be convex. Then for any $(x, y) \in \R^d \times \R^d$, any $g \in \partial\phi(x)$, and any $\eta>0$, 
    \begin{equation*}
        \|y-g\| \leq D_{\eta}\phi(x) + \frac{G_\phi(x, y)}{\eta}.
    \end{equation*}
\end{lemma}
\begin{proof}
    Write $(x_0,y_0)=(x,y)$. The conclusion is immediate if $G_\phi(x_0,y_0)=+\infty$. Otherwise, for any $(x_1,y_1)$ with $\phi^*(y_1)<+\infty$, Fenchel-Young inequalities give
    \begin{equation*}
        \phi(x_0) \geq \langle x_0, y_1\rangle - \phi^*(y_1), \quad \phi^*(y_0) \geq \langle x_1, y_0 \rangle - \phi(x_1),
    \end{equation*}
    so that
    \begin{align*}
        G_\phi(x_0, y_0) \geq& \langle x_0, y_1\rangle - \phi^*(y_1) +  \langle x_1, y_0 \rangle - \phi(x_1) - \langle x_0, y_0 \rangle \\
        =& - G_\phi(x_1, y_1) - \langle x_1 - x_0, y_1 - y_0 \rangle.
    \end{align*}
    In particular, when one takes $(x_1, y_1) \in \graph(\partial\phi)$, $G_\phi(x_1, y_1) = 0$ so that the first term on the right vanishes. Fix any $g_0 \in \partial \phi(x_0)$. If $y_0=g_0$, the result is immediate. Otherwise choose
    \begin{equation*}
        x_1 = x_0 + \eta\frac{y_0 - g_0}{\|y_0 - g_0\|},
    \end{equation*}
    and also choose any $y_1 \in \partial \phi(x_1)$. Then
    \begin{align*}
        G_\phi(x_0, y_0) \geq& \langle x_1 - x_0, y_0 - y_1 \rangle = \eta \Big\langle \frac{y_0 - g_0}{\|y_0 - g_0\|}, y_0 - g_0 + g_0- y_1 \Big\rangle\\
        \geq& \eta \left( \|y_0 - g_0\| - \|g_0 - y_1\|\right).
    \end{align*}
    Since $x_1 \in \overline B_\eta(x_0)$, rearranging,
    \begin{equation*}
        \|y_0 - g_0\| \leq D_\eta\phi(x_0) + \frac{G_\phi(x_0, y_0)}{\eta}.
    \end{equation*}
\end{proof}

\subsection{Coupling transport plans}

To compare probability measures $\gamma_0, \gamma_1 \in \PP(\R^d \times \R^d)$, we use the $q$-Wasserstein distance on $\R^{2d}$ with the norm
\begin{equation*}
   \|(x, y)\| = \left(\|x\|^q + \|y\|^q\right)^{\frac{1}{q}},
\end{equation*}
so that
\begin{equation*}
    W_q^q(\gamma_0, \gamma_1) = \inf_{\pi \in \Pi(\gamma_0, \gamma_1)} \int_{\R^{4d}} \|x_0 - x_1\|^q + \| y_0 - y_1\|^q \di \pi(x_0, y_0, x_1, y_1).
\end{equation*}
Since all norms on finite dimensional spaces are equivalent, this distance is equivalent to the Wasserstein distance induced by the Euclidean norm on $\R^{2d}$, up to multiplicative constants. All our bounds hold for any $q \in [1, \infty)$, with an exponent that degenerates as $q \to \infty$. The examples given in the sharpness section indicate that in general, one cannot hope for any of our bounds to hold for $q= \infty$.

For $\rho_0, \rho_1 \in \PP(\X)$ and $\mu_0, \mu_1 \in \PP(\Y)$, two plans $\gamma_i \in \Pi(\rho_i, \mu_i)$ and a coupling $\beta \in \Pi(\rho_0, \rho_1)$, we denote by $\pi_\beta \in \Pi(\gamma_0, \gamma_1)$ the coupling given by
\begin{equation}
    \label{eq: pi beta defn}
    \pi_\beta(\di x_0, \di y_0, \di x_1, \di y_1) = \beta(\di x_0, \di x_1) \gamma_0(\di y_0 \mid x_0) \gamma_1(\di y_1 \mid x_1),
\end{equation}
where $\gamma_i(\cdot \mid x_i)$ denotes the disintegration of $\gamma_i$ at $x_i$. We will use the coupling $\pi_\beta \in \Pi(\gamma_0, \gamma_1)$ as a competitor for the minimisation problem $W_q(\gamma_0, \gamma_1)$, controlling its transport cost in terms of our objects of interest. In particular,
\begin{align*}
    W_q^q(\gamma_0, \gamma_1) \leq& \int_{\R^{4d}} \|x_0 - x_1\|^q + \| y_0 - y_1\|^q \di \pi_\beta \\
    =& \int_{\R^{2d}}\|x_0 - x_1\|^q \di \beta + \int_{\R^{4d}}\| y_0 - y_1\|^q \di \pi_\beta.
\end{align*}
Choosing $\beta \in \Pi(\rho_0, \rho_1)$ optimal for some $r$-Wasserstein distance, the first term is easily controllable in terms of $W_r(\rho_0, \rho_1)$. The majority of our work will be spent controlling the second term. Observe that when $\rho= \rho_0 = \rho_1$ and the $\gamma_i$ are induced by maps $T_i$, then for $\beta$ the identity coupling,
\begin{equation*}
   \pi_\beta = (\id, T_0, \id, T_1)_\#\rho, \quad  \int_{\R^{4d}}\| y_0 - y_1\|^q \di \pi_\beta = \|T_0 - T_1\|_{L^q(\rho)}^q.
\end{equation*}
The below lemma provides the general framework for all the bounds in this paper. The first bound considers when both plans are concentrated on the same subgradient, and the second when only one is.
\begin{lemma}[Coupling transport plans]
    \label{lem: general coupling of plans}
    Let $q \in [1, +\infty)$, $\X, \Y \subset\R^d$ be compact and convex, and let $\phi : \R^d \to \R$ be a proper convex function with $\partial \phi(\R^d) \subset \Y$. Let $\tilde \rho, \rho_0, \rho_1 \in \PP(\X)$ and  $\mu_0, \mu_1 \in \PP(\Y)$, and suppose for $i=0, 1$ that $\gamma_i \in \Pi(\rho_i, \mu_i)$ satisfy $\spt \gamma_i \subset \graph(\partial\phi)$. Then for any $\alpha \in \Pi(\tilde \rho, \rho_0)$ and $\beta\in \Pi(\rho_0, \rho_1)$ and $\eta,s\geq0$,
    \begin{align*}
            \int_{(\X \times \Y)^2}\|y_0-y_1\|^q\di\pi_\beta  \leq& \int_{\X} D_{s+\eta} \phi(\tilde x)^q \di \tilde \rho(\tilde x)\\
                &+D_\Y^q\alpha(\|\tilde x-x_0\|>\eta) + D_\Y^q\beta(\|x_0-x_1\|>s).
    \end{align*}
    Alternatively, if we do not have the hypothesis $\spt \gamma_1 \subset \graph(\partial\phi)$, then for $s>0$ instead
    \begin{align*}
            \int_{(\X \times \Y)^2}\|y_0-y_1\|^q\di\pi_\beta  \leq& D_\Y^{q-1}\int_{\X} D_{s+\eta} \phi(\tilde x) \di \tilde \rho(\tilde x)\\
                &+ \frac{D_\Y^{q-1}}{s}\int_{(\X \times \Y)^2} G_\phi(x_0, y_1) \di \pi_\beta+D_\Y^q\alpha(\|\tilde x-x_0\|>\eta).
        \end{align*}
\end{lemma}
\begin{proof}
    Partitioning, for $s \geq 0$, over $ \|x_0 - x_1\| >s$ and $\|x_0 - x_1\| \leq s$,
    \begin{equation*}
        \int_{(\X\times \Y)^2}\|y_0-y_1\|^q\di\pi_\beta \leq \int_{\X} D_s \phi(x_0)^q \di \rho_0(x_0) + D_\Y^q \beta(\|x_0 - x_1\|>s).
    \end{equation*}
    The integral term is controlled by partitioning for $\eta\geq 0$, on $\| \tilde x - x_0\| \leq \eta$ or $> \eta$,
    \begin{equation*}
        \int_{\X} D_s \phi(x_0)^q \di \rho_0(x_0) \leq \int_{\X} D_{s+\eta} \phi(\tilde x)^q \di \tilde \rho(\tilde x) + D_\Y^q \alpha(\|x_0 - \tilde x\|>\eta),
    \end{equation*}
    which completes the proof of the first statement. For the second, supposing instead we do not know that $\spt \gamma_1 \subset \graph(\partial \phi)$, we instead apply Lemma \ref{lem: Fenchel bound distance}, to deduce
    \begin{align*}
        \int_{(\X\times \Y)^2}\|y_0-y_1\|^q\di\pi_\beta \leq&
         D_\Y^{q-1} \left( \int_{(\X\times \Y)^2} D_s \phi(x_0)\di\pi_\beta + \frac{1}{s} \int_{(\X\times \Y)^2} G_{\phi}(x_0, y_1) \di \pi_\beta\right).
    \end{align*}
    The first term is controlled as above, giving
    \begin{align*}
        \int_{(\X\times \Y)^2} D_s \phi(x_0)\di\pi_\beta \leq \int_{\X} D_{s+\eta} \phi(\tilde x) \di \tilde \rho(\tilde x) + D_\Y \alpha(\|x_0 - \tilde x\|>\eta),
    \end{align*}
    which completes the proof.
\end{proof}

\section{Quantitative uniqueness: plans on a common subgradient}
\label{ch: quantitative uniqueness}

In this section we prove two theorems regarding the comparison of transport plans concentrated on the subgradient of the same convex function. The first, quantitative uniqueness, compares plans with the same marginals, when one marginal is close to regular. The second compares a plan with a regular first marginal to another plan with potentially different marginals. Although, by the triangle inequality, the second implies the first (with a larger constant), we prefer to prove the two results separately to isolate the proof mechanisms involved.

\subsection{Quantitative uniqueness}

\begin{samepage}
\begin{theorem}[Quantitative uniqueness]
\label{thrm: quant unique general}
Let $q\in[1,\infty)$ and $r \in [1, \infty]$, and let $\X$ and  $\Y$ be compact and convex subsets of $\R^d$ for $d\geq2$. Let $\tilde\rho$ be a probability measure on $\X$ satisfying \ref{eq: assumption eqn}. Then for all $\rho\in\PP(\X)$ and $\mu\in\PP(\Y)$,
\begin{equation}
\label{eq: quant unique}
        \diam_{W_q}\Gamma(\rho,\mu)
        \leq C W_r(\rho,\tilde\rho)^{\nu},
\end{equation}
where $\nu \in (0, 1]$ is given by
\begin{equation*}
   \nu = \begin{cases}
            \frac{ r\theta}{q(r+\theta)} & r<\infty\\
            \frac{\theta}{q}& r= \infty, 
        \end{cases}
\end{equation*}
and $C = D_\Y \left(c_d M D_\X^{d-1+\theta}+ 1\right)^{1/q}D_\X^{-\nu}>0$.
\end{theorem}
\end{samepage}
\begin{proof}
Fix two optimal plans $\gamma_0, \gamma_1 \in \Gamma(\rho, \mu)$. If $W_r(\rho, \tilde \rho)=0$ then $\rho = \tilde \rho$ and the result holds by the Brenier-McCann theorem, so assume that $W_r(\rho, \tilde \rho)>0$, with $\alpha$ the optimal coupling between them. Fix a Brenier potential $\phi$ for the transport from $\rho$ to $\mu$. We apply Lemma \ref{lem: general coupling of plans} with $\beta$ the identity coupling and $s=0$:
\begin{equation*}
    \int\|y_0-y_1\|^q\di\pi_\beta  \leq D_\Y^{q-1} \int_{\X} D_{\eta} \phi(\tilde x) \di \tilde \rho(\tilde x) +D_\Y^q\alpha(\|\tilde x-x_0\|>\eta).
\end{equation*}
We apply Theorem \ref{thrm: quant reg} and a Markov bound. For $\eta>0$,
\begin{equation*}
    \int\|y_0-y_1\|^q\di\pi_\beta \leq D_\Y^q \left( c_d M D_\X^{d-1+\theta} \left(\frac{\eta}{D_\X}\right)^\theta + \frac{W_r^r(\rho, \tilde \rho)}{\eta^r}\right).
\end{equation*}
We balance the terms containing $\eta$, choosing
\begin{equation*}
    \left(\frac{\eta}{D_\X}\right)^\theta = \frac{W_r^r(\rho, \tilde \rho)}{\eta^r}, \quad \text{ so that } \quad \eta = \left(W_r^r(\rho, \tilde \rho) D_\X^\theta\right)^{1/(r+\theta)}.
\end{equation*}
Consequently,
\begin{equation}
    \label{eq: uniqueness bound just before ry opt}
    \int\|y_0-y_1\|^q\di\pi_\beta \leq D_\Y^q \left(c_d M D_\X^{d-1+\theta}+ 1\right)D_\X^{-\frac{r\theta}{r+\theta}}W_r(\rho, \tilde \rho)^{\frac{r\theta}{r + \theta}}.
\end{equation}
This is the claimed bound. For $r=+\infty$, the proof is the same, choosing $\eta= W_\infty(\rho, \tilde \rho)$ without needing any Markov bound.
\end{proof}

\begin{remark}[Uniformity of constant in $q$ and $r$ and $\theta$]
    \label{rmk: uniformity of constant in q r theta}
    The statements of the theorems in the introduction claim the existence of a constant depending on only $\X$, $\Y$ and $M$, and a priori, the above constant
    \begin{equation*}
        C = D_\Y\left(c_d M D_\X^{d-1+\theta}+ 1\right)^{1/q}D_\X^{-\nu}
    \end{equation*}
    does not satisfy this. Since any constant which holds for all measures on $\X'$ also holds for all measures on $\X\subset \X'$, with $D_{\X} \leq D_\X'$, we can freely replace instances of $D_\X$ by $\max(1, D_\X)$; then $\max(1, D_\X)^{-\nu} \leq 1$ and $\max(1, D_\X)^{d-1+\theta} \leq \max(1, D_\X)^d$. Furthermore, since the expression inside the $q$th root is larger than $1$, we can bound this by the same term without the $q$th root. Thus a choice of constant depending only on $\X, \Y$ and $M$ is given by
    \begin{equation*}
        C = D_\Y(1 + c_d M \max(1, D_\X)^d).
    \end{equation*}
    The same principle applies to all the other theorems using similar calculations.
\end{remark}

\subsection{Plans on the same subgradient with different marginals}

The above result uses the $\alpha$ coupling in Lemma \ref{lem: general coupling of plans}. If instead we use the $\beta$ coupling, we can arrive at the following theorem of independent interest, which generalises a result of \cite{cazelles2026statistical} to compare transport plans in the same subgradient with different marginals. We improve their exponent and give a general $W_q-W_r$ bound rather than $W_2-W_2$. Our exponent is proven sharp in Proposition \ref{prop: sharpness quant unique}. We use this theorem in the Appendix to give an alternative proof of stability for the optimal transport plans.
\begin{samepage}
\begin{theorem}[Same subgradient plan control]
\label{thrm: same subgrad plan control}
    Let $q \geq1$ and $r \in [1, \infty]$, let $\X$ and $\Y$ be compact and convex subsets of $\R^d$ for $d\geq2$, and let $\rho_0$ be a probability measure on $\X$ satisfying \ref{eq: assumption eqn}. Let $\phi : \R^d \to \R$ be convex with $\partial \phi(\R^d) \subset \Y$, and set $\gamma_0= (\id, \nabla \phi)_\#\rho_0$. Then for any $\gamma_1 \in \PP(\X \times \Y)$ satisfying $\spt \gamma_1 \subset \graph(\partial \phi)$ with first marginal $\rho_1 \in \PP(\X)$,
    \begin{equation*}
        W_q(\gamma_0, \gamma_1) \leq  C W_r(\rho_0, \rho_1)^\nu,
    \end{equation*}
    where $\nu \in (0, 1]$ is given by
    \begin{equation*}
    \nu = \begin{cases}
                \frac{ r\theta}{q(r+\theta)} & r<\infty\\
                \frac{\theta}{q}& r= \infty, 
            \end{cases}
    \end{equation*}
    and $C^q= D_\X^{q(1-\nu)}+ D_\Y^q(1+ c_d M D_\X^{d-1+\theta}) D_\X^{-q\nu}$.
\end{theorem}
\end{samepage}
\begin{proof}
    If $W_r(\rho_0, \rho_1)=0$, then $\rho_0=\rho_1$, and differentiability of $\phi$ $\rho_0$-almost everywhere implies $\gamma_0=\gamma_1$, so assume that $W_r(\rho_0, \rho_1)>0$.

    We apply Lemma~\ref{lem: general coupling of plans} with $\tilde \rho = \rho_0$, $\rho_0 = \rho_0$, $\rho_1 = \rho_1$, $\alpha\in \Pi(\rho_0, \rho_0)$ the identity plan, $\beta$ optimal for $W_r(\rho_0, \rho_1)$, and $\eta=0$. Then for all $s \geq 0$,
    \begin{align*}
        W_q^q(\gamma_0, \gamma_1) \leq& \int_{\X^2}\|x_0 -x_1\|^q \di \beta(x_0, x_1)\\ &+ \int_\X D_{s}^q \phi(x_0)\di \rho_0(x_0)
        + D_\Y^q\beta(\|x_0 - x_1\|> s).
    \end{align*}
    For the last two terms, Theorem~\ref{thrm: quant reg} and Markov's inequality give, for finite $r$ and $0<s$,
    \begin{align*}
        \int_\X D_{s} \phi(x_0)^q\di \rho_0(x_0)
        + D_\Y^q\beta(\|x_0 - x_1\|> s)\\
        \leq D_\Y^q \left(c_d M D_\X^{d-1}s^\theta + \frac{W_r^r(\rho_0, \rho_1)}{s^r}\right).
    \end{align*}
    As in the quantitative uniqueness theorem, we choose
    \begin{equation*}
        \left(\frac{s}{D_\X}\right)^\theta = \frac{W_r^r(\rho_0, \rho_1)}{s^r}, \quad \text{ so that } \quad s = \left(W_r^r(\rho_0, \rho_1) D_\X^\theta\right)^{1/(r+\theta)},
    \end{equation*}
    which, substituting, gives
    \begin{equation}
        \label{eq: same subgradient diff marg before beta}
        W_q^q(\gamma_0, \gamma_1) \leq \int\|x_0 -x_1\|^q \di \beta + D_\Y^q(1+ c_d M D_\X^{d-1+\theta}) D_\X^{-q\nu} W_r(\rho_0, \rho_1)^{q\nu}.
    \end{equation}
    The same bound holds in the case $r=\infty$ by directly choosing $s= W_\infty(\rho_0, \rho_1)$ without need of the Markov bound. To control the $\beta$ term,
    \begin{equation}
    \label{eq: bound on source coupling control q or r ordered}
        \int_{\X^2} \|x_0 - x_1\|^q \di \beta \leq \begin{cases}
            D_\X^{q-r}W_r^r(\rho_0, \rho_1) & 1 \leq r \leq q <\infty\\
            W_r^q(\rho_0, \rho_1) & 1 \leq q \leq r \leq \infty.
        \end{cases}
    \end{equation}
    Here the first case follows from $\|x_0 - x_1\| \leq D_\X$ and the second by Jensen's inequality if $r<\infty$ or $\beta$ almost everywhere if $r=\infty$. Thus for $q\geq 1$, $r \in [1, \infty]$,
    \begin{equation*}
        \int_{\X^2} \|x_0 - x_1\|^q \di \beta \leq D_\X^{q(1-\nu)} W_r(\rho_0, \rho_1)^{q\nu},
    \end{equation*}
    which, after combining with \eqref{eq: same subgradient diff marg before beta}, gives the result.
\end{proof}

\section{Coercivity and stability: plans close to a common subgradient}

In this section, we prove two coercivity theorems for plans, showcasing two different types of bound. In principle, one could create a combined coercivity theorem containing both, but for readability, we prefer to present each idea separately. We then prove a coercivity theorem for maps with respect to dual suboptimality, followed by the two main stability theorems for maps and plans.

\subsection{Coercivity for plans with arbitrary marginals}

We note that unlike every other theorem in this paper, the constants in the following theorem are not translation invariant with respect to the marginals, as the suboptimality gap between plans with different marginals is not translation invariant.
Indeed, for any $a \in \R^d$ and $\gamma_0, \gamma_1\in\PP(\X \times \Y)$, define
\begin{equation*}
    \gamma_i^a  = (x+ a, y)_\#\gamma_i \in \PP((\X + a) \times \Y).
\end{equation*}
Then
\begin{align*}
    \int_{(\X + a) \times \Y} \langle  x, y \rangle \di ( \gamma_1^a -  \gamma_0^a) =& \int_{\X \times \Y} \langle x + a, y \rangle \di (\gamma_1 - \gamma_0) \\ =& \int_{\X \times \Y} \langle x, y \rangle \di (\gamma_1 - \gamma_0) + \big\langle a, \int_\Y y\di (\mu_1 - \mu_0) \big \rangle, 
\end{align*} 
which can be made to take any value by varying $a$, if $\mu_1$ and $\mu_0$ do not have the same barycentre. Thus some constants here will be given in terms of $R_\X$ and $R_\Y$. We will control the mixed Fenchel deficit appearing in Lemma \ref{lem: general coupling of plans} using the following lemma.
\begin{lemma}
    \label{lem: direct control on the integral fenchel deficit by coupling}
    Let $\X$ and $\Y$ be compact convex subsets of $\R^d$ for $d\geq 2$, and let $\rho_0, \rho_1 \in \PP(\X)$ and $\mu_0, \mu_1 \in \PP(\Y)$. Fix $\phi$ a Brenier potential between $\rho_0$ and $\mu_0$. Then for any $\gamma_0 \in \Gamma(\rho_0, \mu_0)$ optimal and any $\gamma_1 \in \Pi(\rho_1, \mu_1)$,
    \begin{align*}
        \int_{(\X \times \Y)^2} G_\phi(x_0, y_1) \di \pi_\beta \leq& \left|\int_{\X \times \Y} \langle x, y \rangle \di (\gamma_{0} - \gamma_1)\right| + R_\Y W_1(\rho_0, \rho_1) + R_\X W_1(\mu_0, \mu_1),
    \end{align*}
    where $\pi_\beta$ is the coupling defined in \eqref{eq: pi beta defn}, with $\beta$ chosen optimal for $W_1(\rho_0, \rho_1)$.
\end{lemma}
\begin{proof}
    Denote $\psi = \phi^*$. We have
    \begin{align*}
        \int_{(\X \times \Y)^2} G_\phi(x_0, y_1) \di \pi_\beta =& \int \phi \di \rho_0 + \int \psi \di \mu_1 - \int \langle x_0, y_1 \rangle \di \pi_\beta\\
        =& \int \phi \di \rho_0 + \int \psi \di \mu_0 - \int \langle x_1, y_1 \rangle \di \gamma_1 \\   
        &+ \int \langle x_1 - x_0, y_1 \rangle \di \pi_\beta + \int \psi \di (\mu_1 - \mu_0). 
    \end{align*}
    By optimality, $\int \phi \di \rho_0 + \int \psi \di \mu_0 = \int \langle x, y \rangle \di \gamma_{0}$. By the Cauchy-Schwarz and Kantorovich-Rubinstein inequalities,
    \begin{equation*}
        \int \langle x_1 - x_0, y_1 \rangle \di \pi_\beta \leq R_\Y \int \|x_0 - x_1 \| \di \beta, \qquad \int \psi \di (\mu_1 - \mu_0) \leq R_\X W_1(\mu_0, \mu_1),
    \end{equation*}
    which completes the proof.
\end{proof}

\begin{theorem}[Arbitrary marginal coercivity]
    \label{thrm: coercivity general marginals}
    Let $q \geq1$, and let $\X$ and $\Y$ be compact and convex subsets of $\R^d$ for $d \geq 2$. Let $\rho_0$ be a probability measure on $\X$ satisfying \ref{eq: assumption eqn}, let $\mu_0 \in \PP(\Y)$, and denote by $\gamma_0$ the unique optimal plan between them. Then for any $\rho_1 \in \PP(\X)$, $\mu_1 \in \PP(\Y)$ and any plan $\gamma_1 \in \Pi(\rho_1, \mu_1)$,
    \begin{equation*}
        W_q(\gamma_0, \gamma_1) \leq C \left(  \left|\int_{\X \times \Y} \langle x, y \rangle \di (\gamma_0 - \gamma_1)\right|+ R_\Y W_1(\rho_0, \rho_1) + R_\X W_1(\mu_0, \mu_1) \right)^\frac{\theta}{q(1+\theta)}
    \end{equation*}
    where $C^q = D_\X^{q-1 +1/(1+\theta)}R_\Y^{-\theta/(1+\theta)}+ D_\Y^q(1 + c_d M R_\X^{d-1+ \theta}) (D_\Y R_\X)^{- \theta/((1+\theta))}$.
\end{theorem}
\begin{proof}
    Let $\phi$ be a Brenier potential for the transport between $\rho_0$ and $\mu_0$. We apply Lemma~\ref{lem: general coupling of plans} with $\tilde \rho = \rho_0$, $\rho_0 = \rho_0$, $\rho_1 = \rho_1$, $\alpha$ the identity coupling, and $\beta \in \Pi(\rho_0, \rho_1)$ optimal for $W_1$. Then
    \begin{align*}
        W_q^q(\gamma_0, \gamma_1) \leq& \int_{\X^2} \|x_0 - x_1\|^q \di \beta + \int_{(\X\times \Y)^2} \|y_0 - y_1\|^q \di \pi_\beta\\
        \leq& D_\X^{q-1}W_1(\rho_0, \rho_1) + D_\Y^{q-1}\int_{(\X \times \Y)^2} \|y_0 - y_1\| \di \pi_\beta.
    \end{align*}
    Applying Lemma \ref{lem: Fenchel bound distance} with $x=x_0$, $y=y_1$ and $g=y_0$, followed by Theorem \ref{thrm: quant reg}, for $s>0$,
    \begin{align}
        \int_{(\X \times \Y)^2} \|y_0 - y_1\| \di \pi_\beta \leq D_\Y c_d M R_\X^{d-1} s^\theta + \frac{\Delta}{s},
    \end{align}
    where
    \begin{equation*}
        \Delta = \int G_\phi(x_0, y_1) \di \pi_\beta.
    \end{equation*}
    If $\Delta =0$, passing $s \downarrow 0$ gives the result. Otherwise for $\Delta>0$ we balance the terms in $s$ on the unitless length scale, choosing
    \begin{equation*}
        \left(\frac{s}{R_\X}\right)^\theta = \frac{\Delta}{D_\Y s} \implies s^{\theta+1} = \frac{\Delta R_\X^{\theta}}{D_\Y},
    \end{equation*}
    so that 
    \begin{equation*}
        \frac{\Delta}{s} = R_\X^{-\theta/(1+\theta)}D_\Y^{1/(1+\theta)} \Delta^{\theta/(1+\theta)}, \qquad R_\X^{d-1}s^\theta = R_\X^{d-1 + \theta} R_\X^{-\theta/(1+\theta)} D_\Y^{-\theta/(1+\theta)} \Delta^{\theta/(1+ \theta)}.
    \end{equation*}
    Substituting gives
    \begin{equation*}
        \int_{(\X \times \Y)^2} \|y_0 - y_1\|^q  \di \pi_\beta \leq D_\Y^{q}(1 + c_d M R_\X^{d-1+ \theta}) (D_\Y R_\X)^{- \theta/(1+\theta)}\Delta^{\theta/(1+\theta)}.
    \end{equation*}
    The proof is completed by applying Lemma \ref{lem: direct control on the integral fenchel deficit by coupling} above, and absorbing the $W_1$ term into the constant via
    \begin{equation*}
        D_\X^{q-1}W_1(\rho_0, \rho_1) \leq D_\X^{q-1 +1/(1+\theta)} R_\Y^{-\theta/(1+\theta)} \Big(R_\Y W_1(\rho_0, \rho_1) \Big)^{\theta/(1+\theta)}.
    \end{equation*}
\end{proof}

\begin{remark}[Recovering the version in the introduction]
    \label{rem: intro coercivity}
    Since we are comparing measures with different marginals, the max correlation suboptimality gap above and the quadratic cost optimality gap are not equal. We nonetheless have
    \begin{align*}
        \int_{\X \times \Y} 2\langle x,y\rangle \di(\gamma_0-\gamma_1)
        =& \int_{\X \times \Y}\|x-y\|^2\di(\gamma_1-\gamma_0)\\
        &+\int_\X\|x\|^2\di(\rho_0-\rho_1)+\int_\Y\|y\|^2\di(\mu_0-\mu_1)\\
        \leq& \int_{\X \times \Y}\|x-y\|^2\di(\gamma_1-\gamma_0)\\
        &+2R_\X W_1(\rho_0,\rho_1)+2R_\Y W_1(\mu_0,\mu_1).
    \end{align*}
    After doing the same bound for $\gamma_1 - \gamma_0$, we can control the absolute value, and thus recover the version presented in the introduction, up to a different constant.
\end{remark}

\subsection{Approximate coercivity with an almost regular source measure}

The second plan coercivity theorem concerns the transport problem where one marginal $\rho$ satisfies \ref{eq: assumption eqn} only approximately, similar to the quantitative uniqueness theorem. In this case, we prove that for any $\mu$, the $\rho \to \mu$ transport problem is coercive up to an error given in terms of $W_r(\rho, \tilde \rho)$. While for $r=1$ this theorem is implied by the above coercivity theorem combined with the stability theorem, for other $r$ the error exponent here is sharper. Since we only treat the fixed marginal case, we are once again in the setting where constants are translation invariant, so can be stated in terms of the diameters.

\begin{theorem}[Coercivity with almost regular source]
    \label{thrm: coercivity rough measures}
    Let $q \geq 1$ and $r \in [1, \infty]$, and let $\X$ and $\Y$ be compact and convex subsets of $\R^d$ for $d \geq 2$. Fix $\tilde \rho$ a probability measure on $\X$ satisfying \ref{eq: assumption eqn}. Then for any $\rho \in \PP(\X)$ and $\mu \in \PP(\Y)$ and any optimal plan $\gamma_{\rho \to \mu} \in \Gamma(\rho, \mu)$,
    \begin{equation*}
        \forall \gamma \in \Pi(\rho, \mu), \quad W_q(\gamma_{\rho \to \mu}, \gamma) \leq C_0\left( \int_{{\X \times \Y}} \|x - y\|^2 \di \gamma - W_2^2(\rho, \mu) \right)^{\frac{\theta}{q(1+\theta)}} + C_1 W_r(\rho, \tilde \rho)^\nu,
    \end{equation*}
    where $\nu \in (0, 1]$ is given by
    \begin{equation*}
    \nu = \begin{cases}
                \frac{ r\theta}{q(r+\theta)} & r<\infty\\
                \frac{\theta}{q}& r= \infty, 
            \end{cases}
    \end{equation*}
    where $C_0^q = D_\Y^q(1 + c_d M D_\X^{d-1+ \theta}) (D_\Y D_\X)^{- \theta/(1+\theta)}$ and $C_1>0$ is the constant from Theorem \ref{thrm: quant unique general}.
\end{theorem}
\begin{proof}
       Let $\phi$ be a Brenier potential for the transport $\rho \to \mu$. We apply Lemma \ref{lem: general coupling of plans} with $\tilde \rho = \tilde \rho$, $\rho_0 = \rho_1 = \rho$, $\alpha$ optimal for $W_r(\rho, \tilde \rho)$, and $\beta$ the identity coupling. Then
    \begin{equation*}
        W_q^q(\gamma_0, \gamma_1) \leq \int_{(\X\times \Y)^2} \|y_0 - y_1\|^q \di \pi_\beta \leq D_\Y^{q-1}\int_{(\X \times \Y)^2} \|y_0 - y_1\| \di \pi_\beta.
    \end{equation*}
    For $r<\infty$, applying Lemma \ref{lem: Fenchel bound distance} followed by Theorem \ref{thrm: quant reg}, the subadditivity of $(\cdot)^\theta$, and a Markov bound,
    \begin{align*}
        \int_{(\X \times \Y)^2} \|y_0 - y_1\| \di \pi_\beta \leq& \int_\X D_{\eta+ s} \phi(\tilde x) \di \tilde \rho(\tilde x) + \frac{\Delta}{s} + D_\Y\alpha(\|\tilde x - x\| > \eta)\\
        \leq& D_\Y c_d M D_\X^{d-1} ( s^\theta + \eta^\theta)+ \frac{\Delta}{s} + D_\Y\frac{W_r^r(\rho, \tilde \rho)}{\eta^r}
    \end{align*}
    where as before
    \begin{equation*}
        \Delta = \int G_\phi(x_0, y_1) \di \pi_\beta.
    \end{equation*}
    When $\Delta = 0$ or $W_r(\tilde \rho, \rho)=0$ we collapse into cases covered by the previous theorems. Assuming both are positive, we choose $s$ as in the previous theorem, and $\eta$ as in the quantitative uniqueness theorem, which gives
    \begin{equation*}
        W_q^q(\gamma_0, \gamma_1) \leq C_0^q \Delta^{\theta/(1+\theta)} + C_1^q W_r(\rho, \tilde \rho)^{q\nu}
    \end{equation*}
    for the constants given in the statement of the theorem. We take the $q$th root then use subadditivity, before applying Lemma \ref{lem: direct control on the integral fenchel deficit by coupling} to complete the proof when $r<\infty$. Here we recover the quadratic cost gap directly since the plans have the same marginals, so that
    \begin{equation*}
        \int \langle x, y \rangle \di (\gamma_0 - \gamma_1) = \frac{1}{2}\int\|x-y\|^2 \di (\gamma_1 - \gamma_0).
    \end{equation*}
    When $r=\infty$, the same bounds hold directly, taking the same choice of $s$ and instead $\eta = W_\infty(\rho, \tilde \rho)$ without needing any Markov bound.
\end{proof}

\subsection{Coercivity of maps with respect to dual suboptimality}

For compact sets $\X, \Y$, we denote the maximum correlation semidual functional by $\K_{\rho, \mu} : \mathcal{C}^0(\Y) \to \R$,
\begin{equation}
    \label{eq: k rho mu definition}
    \K_{\rho, \mu}(\psi) = \int_\X \psi^* \di \rho + \int_\Y \psi \di \mu,
\end{equation} 
where for $\psi \in \mathcal{C}^0(\Y)$, $\psi^*(x) = \sup_{y \in \Y} \langle x, y \rangle - \psi(y)$ denotes the Legendre transform of $\psi$ extended to $+\infty$ outside of $\Y$. The following theorem provides a hybrid dual to primal coercivity theorem. It serves to directly extend \cite[Theorem 3.2]{merigot2026sharp}, which corresponds to the case $q=2$ and $\theta = 1$.
\begin{theorem}[Map coercivity with respect to dual suboptimality]
\label{thrm: map coercivity wrt dual suboptimality}
Let $q \in [1, +\infty)$, and let $\X$ and $\Y$ be compact and convex subsets of $\R^d$ for $d \geq2$. Let $\rho$ be a probability measure on $\X$ satisfying \ref{eq: assumption eqn}, and $\mu \in \PP(\Y)$. Then
\begin{equation*}
    \forall \psi \in \mathcal{C}^0(\Y),\quad\| \nabla \psi^* - T_{\rho \to \mu}\|_{L^q(\rho)}^q \leq C\left( \K_{\rho, \mu}(\psi) - \min_{\mathcal{C}^0(\Y)} \K_{\rho, \mu} \right)^{\frac{\theta}{1+\theta}},
\end{equation*}
where $C= D_\Y^{q}(1 + c_d M D_\X^{d-1+ \theta}) (D_\Y D_\X)^{- \theta/(1+\theta)}$.
\end{theorem}
\begin{proof}
    Denote $T = T_{\rho \to \mu}$. Since $\partial\psi^*(\X) \subset \Y$, we apply Lemma \ref{lem: general coupling of plans} with $\rho_0 = \rho_1 = \tilde \rho = \rho$, $\gamma_0 = (\id, \nabla\psi^*)_\#\rho$, $\gamma_1 = (\id, T)_\#\rho$, $\phi = \psi^*$, $\alpha, \beta$ the identity couplings, and $\eta= 0$. For this choice of coupling,
    \begin{equation*}
        \int_{(\X \times \Y)^2} \|y_0 - y_1\|^q \di \pi_\beta = \|\nabla\psi^* - T\|_{L^q(\rho)}^q.
    \end{equation*}
    The lemma implies that for all $s>0$
    \begin{equation*}
        \|\nabla\psi^* - T\|_{L^q(\rho)}^q \leq D_\Y^{q-1}\left(\int_\X D_s \psi^*(x) \di \rho(x) + \frac{\Delta}{s} \right),
    \end{equation*}
    where
    \begin{equation*}
        \Delta = \int_\X G_{\psi^*}(x, T(x)) \di \rho(x).
    \end{equation*}
    Applying Theorem~\ref{thrm: quant reg}, for all $s>0$,
 \begin{equation*}
        \|\nabla\psi^* - T\|_{L^q(\rho)}^q \leq D_\Y^{q-1}\left(D_\Y c_d M D_\X^{d-1} s^\theta + \frac{\Delta}{s} \right).
    \end{equation*}
    If $\Delta = 0$, letting $s\downarrow 0$ completes the proof. Otherwise, balancing the terms in $s$ as in previous theorems, we choose
    \begin{equation*}
        \left(\frac{s}{D_\X}\right)^\theta = \frac{\Delta}{D_\Y s} \implies s = \left(\frac{\Delta D_\X^\theta}{D_\Y}\right)^{1/(1+\theta)}.
    \end{equation*}
    Substituting this value gives
    \begin{equation*}
        \|\nabla\psi^* - T\|_{L^q(\rho)}^q \leq D_\Y^{q}(1 + c_d M D_\X^{d-1+ \theta}) (D_\Y D_\X)^{- \theta/(1+\theta)}\Delta^{\theta/(1+\theta)}.
    \end{equation*}
    The proof is completed by observing that $\Delta$ is a lower bound for the dual suboptimality gap. Since $\psi^{**} \leq \psi$,
    \begin{align*}
        \int_{\X} G_{\psi^*}(x, T(x)) \di \rho(x) \leq& \int_\X \psi^*(x) + \psi(T(x)) - \langle x, T(x)\rangle \di \rho(x)\\
        =& \K_{\rho, \mu}(\psi)  - \int_\X\langle x, T(x)\rangle \di \rho(x)\\
        =& \K_{\rho, \mu}(\psi) - \min_{\mathcal{C}^0(\Y)} \K_{\rho, \mu},
    \end{align*}
    by optimality of $T$ and strong duality.
\end{proof}

\subsection{Stability of optimal maps}

The following theorem provides a stability bound for optimal transport maps with respect to perturbation of both marginals, again extending a result \cite[Theorem 3.3]{merigot2026sharp} to general exponents $q$ and $\theta$, as well as improving the numerical constant.

A priori, there are some issues with regard to what one means by a bi-marginal stability result for optimal transport maps. In general, as soon as both marginals are perturbed, a Brenier map $T_{\rho_1 \to \mu_1}$ need not exist, since $\partial \phi_{\rho_1 \to \mu_1}$ may be multivalued on a set of positive $\rho_1$ measure. Here we compare the gradients of the Brenier potentials with respect to $L^q(\rho_0)$, using that the gradient of any convex function is well defined $\rho_0$ almost everywhere.

Nonetheless, given a Brenier potential for the transport from $\rho_1$ to $\mu_1$, there is no canonical choice for what values it should take outside the support of $\rho_1$, and so the gradients of two different Brenier potentials, although they are both well defined, might differ on a set of positive $\rho_0$ measure. The below theorem applies to \textit{any} Brenier potential which satisfies \eqref{eq: double legendre trans eq}.

\begin{theorem}[Bi-marginal Brenier gradient stability]
\label{thrm: bi marginal map stability}
Let $q \in [1, +\infty)$, let $\X$ and $\Y$ be compact subsets of $\R^d$, and let $\rho_0 \in \PP(\X)$ satisfy \ref{eq: assumption eqn}. Then for all $\rho_1 \in \PP(\X)$ and $\mu_0, \mu_1 \in \PP(\Y)$, and Brenier potentials $\phi_{\rho_i \to \mu_i}$
\begin{equation*}
    \|\nabla \phi_{\rho_0 \to \mu_0} - \nabla \phi_{\rho_1 \to \mu_1}\|_{L^q(\rho_0)}^q \leq C \Big(D_\Y W_1(\rho_0, \rho_1) +D_\X W_1(\mu_0, \mu_1)\Big)^{\frac{\theta}{1+\theta}}
\end{equation*}
where $C$ is the constant in Theorem~\ref{thrm: map coercivity wrt dual suboptimality}.
\end{theorem}
\begin{proof}
Let $(\phi_0, \psi_0)$ and $(\phi_1, \psi_1)$ be pairs of optimal potentials for the transport between $\rho_i$ and $\mu_i$ respectively. We apply the dual coercivity Theorem~\ref{thrm: map coercivity wrt dual suboptimality} to the $\rho_0, \mu_0$ problem, using that $(\phi_1, \psi_1)$ is a competitor to the dual $\K_{\rho_0, \mu_0}$. Thus
\begin{equation}
    \label{eq: bi marignal map subopt gap}
    \| \nabla \phi_{\rho_0 \to \mu_0} - \nabla \phi_{\rho_1 \to \mu_1}\|_{L^q(\rho_0)}^q \leq C \left( \int \phi_1 \di \rho_0 + \int \psi_1 \di \mu_0 - \int \phi_0 \di \rho_0  -\int \psi_0 \di \mu_0 \right)^{\frac{\theta}{1+\theta}}.
\end{equation}
Adding the second dual suboptimality gap
\begin{equation*}
    \int \phi_0 \di \rho_1  + \int \psi_0 \di \mu_1 - \int \phi_1 \di \rho_1  - \int \psi_1 \di \mu_1  \geq 0
\end{equation*}
to the right hand side of \eqref{eq: bi marignal map subopt gap} inside the bracket,
\begin{equation*}
    \| \nabla \phi_{\rho_0 \to \mu_0} - \nabla \phi_{\rho_1 \to \mu_1}\|_{L^q(\rho_0)}^q \leq C \left( \int (\phi_1 - \phi_0) \di(\rho_0-\rho_1) + \int (\psi_1 - \psi_0) \di(\mu_0 - \mu_1) \right)^{\frac{\theta}{1+\theta}}.
\end{equation*}
Since $\nabla \phi_i \subset \Y$ and $\nabla \psi_i \subset \X$ for $i=0, 1$, then the functions $\phi_1 - \phi_0$ and $\psi_1 - \psi_0$ are $D_\Y$ and $D_\X$ Lipschitz respectively, so the result follows by Kantorovich-Rubinstein duality.
\end{proof}

\begin{remark}[Interpretation for statistical estimation of transport maps]
    \label{rmk: statistical estimation of transport maps}
The theorem has the following implications for approximation of transport maps. It says that if one can solve the dual problem for $\rho_1, \mu_1$ to get some Brenier potential $\phi_{\rho_1 \to \mu_1}$ defined on all $\X$, then $\nabla \phi_{\rho_1 \to \mu_1}$ is a good approximator in $L^q(\rho_0)$ for the transport map $T_{\rho_0 \to \mu_0} = \nabla \phi_{\rho_0 \to \mu_0}$. This means that a map constructed out of the gradient of the Brenier potential $\phi_{\rho_1 \to \mu_1}$ will, on average, send samples from $\rho_0$ to a point near the image of the true optimal map $T_{\rho_0 \to \mu_0}$.
\end{remark}

The fixed source stability of optimal transport maps presented in the introduction is a direct consequence of the above version by taking $\rho_0 = \rho_1$, since $T_{\rho \to \mu_i} = \nabla \phi_{\rho \to \mu_i}$.

\subsection{Stability of optimal plans}

We now prove the bi-marginal stability of optimal transport plans with respect to perturbation of both marginals. This is a relatively direct consequence of Theorem \ref{thrm: coercivity general marginals}. In the appendix, we present a second proof of this theorem using Theorem \ref{thrm: bi marginal map stability} and Theorem \ref{thrm: same subgrad plan control}.

\begin{theorem}[Bi-marginal stability of optimal plans]
    \label{thrm: bi marginal plan stability}
    Let $q \in [1, +\infty)$, and let $\X$ and $\Y$ be compact convex subsets of $\R^d$ for $d \geq 2$. Let $\rho_0$ be a probability measure on $\X$ satisfying \ref{eq: assumption eqn}, let $\mu_0 \in \PP(\Y)$, and let $\gamma_{\rho_0 \to \mu_0} \in \Gamma(\rho_0, \mu_0)$ be the unique optimal plan between them. Then for all $\rho_1 \in \PP(\X)$, all $\mu_1 \in \PP(\Y)$ and any optimal plan $\gamma_{\rho_1 \to \mu_1} \in \Gamma(\rho_1, \mu_1)$,
    \begin{equation*}
        W_q(\gamma_{\rho_0 \to \mu_0}, \gamma_{\rho_1 \to \mu_1}) \leq C\Big( W_1(\rho_0, \rho_1)+  W_1(\mu_0, \mu_1)\Big)^{\frac{\theta}{q(1+\theta)}},
    \end{equation*}
    where $C = 2\max(D_\X, D_\Y)^{\theta/(q(1+\theta))}C'$ for $C'$ the constant in Theorem \ref{thrm: coercivity general marginals}, with $R_\X$ and $R_\Y$ replaced by $D_\X$ and $D_\Y$.
\end{theorem}
\begin{proof}
    Since all the quantities in the theorem are invariant under translation of source and target domains, we can assume without loss of generality that $0 \in \X \cap \Y$. We apply the arbitrary marginal coercivity theorem, Theorem \ref{thrm: coercivity general marginals}, with $\gamma_i = \gamma_{\rho_i \to \mu_i}$, so that
    \begin{equation}
        \label{eq: stability bound pre coercivity}
        W_q(\gamma_0, \gamma_1)\leq C' \left( \left|\int_{\X \times \Y} \langle x, y \rangle \di(\gamma_0 - \gamma_1)\right| + D_\Y W_1(\rho_0, \rho_1) + D_\X W_1(\mu_0, \mu_1) \right)^{\frac{\theta}{q(1+\theta)}},
    \end{equation}
    and we replaced the instances of $R_\X$ and  $R_\Y$ by $D_\X$ and  $D_\Y$, since $0 \in \X \cap \Y$. Let $(\phi_0, \psi_0)$ and $(\phi_1, \psi_1)$ be pairs of optimal maximum correlation potentials for the respective dual problems from $\rho_i$ to $\mu_i$. Write $\gamma_i=\gamma_{\rho_i\to\mu_i}$. The compatibility condition for each of the two transport problems implies
    \begin{align*}
        \int \langle x,y\rangle\di(\gamma_0-\gamma_1)
        =& \int\phi_0\di\rho_0+\int\psi_0\di\mu_0-\int\phi_1\di\rho_1-\int\psi_1\di\mu_1\\
        \leq& \int\phi_1\di(\rho_0-\rho_1)+\int\psi_1\di(\mu_0-\mu_1)\\
        \leq& D_\Y W_1(\rho_0,\rho_1)+D_\X W_1(\mu_0,\mu_1),
    \end{align*}
    where we used the suboptimality of $(\phi_1,\psi_1)$ in the $\rho_0$ to $\mu_0$ dual problem, followed by Kantorovich-Rubinstein duality. Interchanging $\gamma_0$ and $\gamma_1$ and applying a symmetric bound controls the absolute value $\left|\int \langle x, y \rangle \di (\gamma_0 - \gamma_1)\right|$. Thus the gap
    \begin{equation*}
        \Delta = \left|\int_{\X \times \Y} \langle x, y \rangle \di (\gamma_0 - \gamma_1) \right| + D_\Y W_1(\rho_0, \rho_1) + D_\X W_1(\mu_0, \mu_1)
    \end{equation*}
    in Theorem~\ref{thrm: coercivity general marginals} satisfies
    \begin{equation*}
        \Delta\leq 2D_\Y W_1(\rho_0,\rho_1)+2D_\X W_1(\mu_0,\mu_1) \leq 2\max(D_\X, D_\Y)(W_1(\rho_0, \rho_1) + W_1(\mu_0, \mu_1)).
    \end{equation*}
    Substituting this bound into \eqref{eq: stability bound pre coercivity} completes the proof.
\end{proof}

\section{Sharpness}

In this section we provide two examples which prove that the four main theorems are sharp in exponent, in the sense that the same bound cannot hold with a better exponent. The first example provides sharpness for quantitative uniqueness as well as the same subgradient plan control, while the second provides sharpness for the stability and coercivity of plans, as well as stability of maps. The second example is heavily based on the example given in \cite{letrouit2026unstable}, which corresponds to the case $\theta = 1$.

\subsection{Sharpness for uniqueness}

In general, it is challenging to ascertain a lower bound on $W_q(\gamma_0, \gamma_1)$, since often the optimal coupling between the plans is not clear. This is in contrast to lower bounding a distance between maps $\|T_0 - T_1\|_{L^q(\rho)}$, where as soon as the optimal map is known explicitly, the distance can be calculated. For this first example, it is sufficient for us to use the crude bound
\begin{equation*}
        W_q^q(\gamma_0, \gamma_1)  \geq \int \dist( (x, y), \spt\gamma_0)^q \di \gamma_1(x, y),
\end{equation*}
as the expression on the right is more easily calculable. This bound holds since, at best, every point in $\spt\gamma_1$ is coupled with its closest point to $\spt \gamma_0$.
\begin{proposition}
\label{prop: sharpness quant unique}
    Let $d \geq 2$ and $\theta \in (0, 1]$. There exists a probability measure $\tilde  \rho$ on $[-1, 1]^d$ satisfying \ref{eq: assumption eqn} for some $M>0$, and families of probability measures 
    \begin{equation*}
        (\rho_\varepsilon)_{\varepsilon \in [0, 1]}, (\mu_\varepsilon)_{\varepsilon \in [0, 1]} \in \PP([-1, 1]^d),
    \end{equation*}
    such that for all $q\geq1$ and $r \in [1, \infty]$,
    \begin{equation*}
        \diam_{W_q} \Gamma(\rho_\varepsilon, \mu_\varepsilon) \geq 8^{-1} \varepsilon^{\frac{\theta}{q}}, \quad \text{ but } \quad W_r(\rho_\varepsilon, \tilde \rho) \leq \begin{cases}
        \varepsilon^{\frac{r+\theta}{r}} & r<\infty, \\
        \varepsilon & r = \infty.
        \end{cases}
    \end{equation*}
    Consequently, Theorems~\ref{thrm: quant unique general} and \ref{thrm: same subgrad plan control} are optimal in exponent for all $q, r, \theta$.
\end{proposition}
\begin{proof}
    We choose
    \begin{equation*}
        \di \tilde \rho(x) = \frac{\theta}{2^d} |x_1|^{\theta-1} \chi_{[-1,1]^d}(x)\di x
    \end{equation*}
    as the uniform measure weighted by an integrable singularity $|x_1|^{\theta - 1}$. Then for any $B_\eta(z)$, denoting by $B_\eta(z_{-1})$ the ball in $\R^{d-1}$ centred at the point $(z_2,...,z_d)$,
    \begin{align*}
        \tilde \rho(B_\eta(z)) \leq& \tilde \rho( [z_1-\eta, z_1+\eta] \times B_\eta(z_{-1})) \\\leq& \eta^{d-1} \beta_{d-1} \frac{\theta}{2^d}\int_{z_1-\eta}^{z_1+\eta} |x_1|^{\theta - 1} \di x_1  \leq \frac{\beta_{d-1}}{2^{d-1}}\eta^{d-1+ \theta}  
    \end{align*}
    since the integral is maximised when $z_1 = 0$. Consequently $\tilde \rho$ satisfies \ref{eq: assumption eqn} with $M = \beta_{d-1}/2^{d-1}$. We define $\rho_\varepsilon$ by collapsing the $\varepsilon$ neighbourhood around the hyperplane $x_1 = 0$, so that $\rho_\varepsilon = P_{\varepsilon\#}\tilde \rho$ where
    \begin{equation*}
        P_\varepsilon(x) = \begin{cases}
            x - x_1 e_1 & \text{ if } |x_1| \leq \varepsilon \\
            x & \text{ otherwise.}
        \end{cases}
    \end{equation*}
    Take target measures
    \begin{equation*}
        \mu_\varepsilon = \frac{1- \varepsilon^\theta}{2} (\delta_{(1, 0,...)} + \delta_{(-1, 0,...)}) + \frac{\varepsilon^\theta}{2}(\delta_{(1/2, 0,...)} + \delta_{(-1/2, 0,...)}).
    \end{equation*}
    For any $\varepsilon>0$, consider for $i=0, 1$ the maps
    \begin{equation*}
        T_i^\varepsilon(x) = \begin{cases}
            \sign(x_1)e_1 &\text{ if } x_1 \neq 0 \\
            \frac{(-1)^i}{2} \sign(x_2)e_1 & \text{ if } x_1 =0,
        \end{cases}
    \end{equation*}
    and induced plans $\gamma_i^\varepsilon = (\id, T_i^\varepsilon)_{\#}\rho_\varepsilon \in \Pi(\rho_\varepsilon, \mu_\varepsilon)$; see Figure~\ref{fig: quant unique example}.
    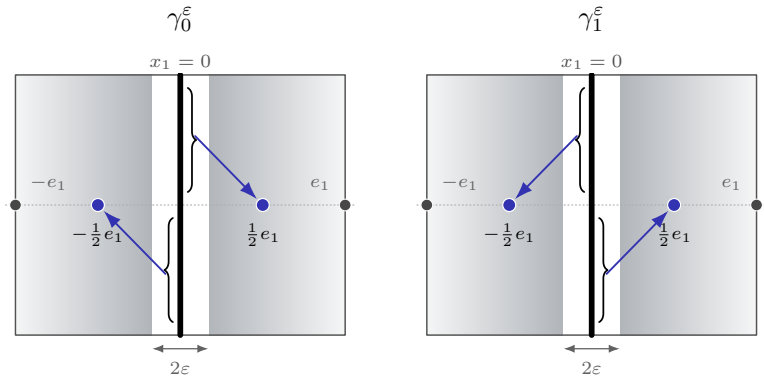
\begin{figure}[ht]
        \centering
            \begin{tikzpicture}[
    x=2.18cm,
    y=1.72cm,
    >=Latex,
    line cap=round,
    line join=round,
    every node/.style={font=\small},
    flow/.style={-{Latex[length=2.4mm,width=1.65mm]},line width=.72pt,targetblue},
    massbrace/.style={decorate,decoration={brace,amplitude=3.2pt},line width=.66pt},
    target/.style={circle,fill=targetblue,draw=white,line width=.45pt,minimum size=4.9pt,inner sep=0pt},
    outertarget/.style={circle,fill=black!72,draw=white,line width=.4pt,minimum size=4.4pt,inner sep=0pt}
    ]

    \def\epsstrip{0.17}

    \newcommand{\sourcepanel}[1]{%
    \shade[left color=sourcegray!6,right color=sourcegray!48]
        (-1,-1) rectangle (-\epsstrip,1);
    \shade[left color=sourcegray!48,right color=sourcegray!6]
        (\epsstrip,-1) rectangle (1,1);
    \fill[white] (-\epsstrip,-1) rectangle (\epsstrip,1);
    \draw[black!68,line width=.55pt] (-1,-1) rectangle (1,1);

    \draw[black!32,densely dotted,line width=.42pt] (-1.06,0)--(1.06,0);
    \draw[black,line width=2.15pt] (0,-1)--(0,1);

    \node[outertarget] at (-1,0) {};
    \node[outertarget] at ( 1,0) {};
    \node[target] at (-.5,0) {};
    \node[target] at ( .5,0) {};

    \node[font=\scriptsize,black!62,above right=2.1pt and 2pt] at (-1,0) {$-e_1$};
    \node[font=\scriptsize,black!62,above left=2.1pt and 2pt] at ( 1,0) {$e_1$};
    \node[font=\scriptsize,below=3.5pt] at (-.5,0) {$-\tfrac12 e_1$};
    \node[font=\scriptsize,below=3.5pt] at ( .5,0) {$\tfrac12 e_1$};

    \node[font=\scriptsize,black!68,fill=white,inner sep=1.2pt]
        at (0,1.10) {$x_1=0$};
    \draw[<->,black!62,line width=.48pt,{Latex[length=1.45mm]}-{Latex[length=1.45mm]}]
        (-\epsstrip,-1.105)--(\epsstrip,-1.105)
        node[midway,below=1.5pt,font=\scriptsize] {$2\varepsilon$};
    \node[font=\normalsize] at (0,1.43) {#1};
    }

    \begin{scope}[xshift=-2.72cm]
    \sourcepanel{$\gamma_0^\varepsilon$}
    \draw[massbrace, decoration={brace,mirror,amplitude=3.2pt}] (.045,.10)--(.045,.90);
    \draw[massbrace,decoration={brace,amplitude=3.2pt}] (-.045,-.90)--(-.045,-.10);
    \draw[flow,shorten >=3pt] (.09,.53)--(.5,0);
    \draw[flow,shorten >=3pt] (-.09,-.53)--(-.5,0);
    \end{scope}

    \begin{scope}[xshift=2.72cm]
    \sourcepanel{$\gamma_1^\varepsilon$}
    \draw[massbrace,decoration={brace,mirror,amplitude=3.2pt}] (-.045,.90)--(-.045,.10);
    \draw[massbrace] (.045,-.10)--(.045,-.90);
    \draw[flow,shorten >=3pt] (-.09,.53)--(-.5,0);
    \draw[flow,shorten >=3pt] (.09,-.53)--(.5,0);
    \end{scope}

            \end{tikzpicture}
        \caption{The optimal transport plans $\gamma_i^\varepsilon$ between $\rho_\varepsilon$ and $\mu_\varepsilon$.}
        \label{fig: quant unique example}
    \end{figure}
    
    Both plans are concentrated on the subgradient of the convex function $\phi(x) = |x_1|$, so that $\gamma_i^\varepsilon \in \Gamma(\rho_\varepsilon, \mu_\varepsilon)$ are optimal. Comparing the marginal measures, for finite $r$, by monotonicity
    \begin{equation*}
        W_r^r(\tilde \rho, \rho_\varepsilon) = \frac{\theta}{2}\int_{-\varepsilon}^\varepsilon |x_1|^{r+ \theta-1} \di x_1 = \frac{\theta}{r+\theta} \varepsilon^{r+\theta}.
    \end{equation*}
    The coupling induced by $P_\varepsilon$ gives $W_\infty(\tilde \rho, \rho_\varepsilon) \leq \varepsilon$. To compare the plans, the distance can be lower bounded by
    \begin{equation*}
        W_q^q(\gamma_0^\varepsilon, \gamma_1^\varepsilon)  \geq \int_{[-1, 1]^{2d}} \dist( (x, y), \spt\gamma_0^\varepsilon)^q \di \gamma_1^\varepsilon(x, y).
    \end{equation*}
    The set $(\{0\} \times [1/2, 1] \times [-1, 1]^{d-2}) \times (\{(-1/2, 0)\} \times [-1, 1]^{d-2})$ has $\gamma_1^\varepsilon$ mass $\varepsilon^\theta/4$, and every point in the set is at least distance $1/2$ from $\spt \gamma_0^\varepsilon$, so that
    \begin{equation*}
        W_q^q(\gamma_0^\varepsilon, \gamma_1^\varepsilon) \geq 2^{-2-q}\varepsilon^\theta, \quad \text{ implying } \quad W_q(\gamma_0^\varepsilon, \gamma_1^\varepsilon) \geq 2^{-3} \varepsilon^{\theta/q}.
    \end{equation*}
    It follows that Theorem \ref{thrm: quant unique general} cannot hold in general with a better exponent.

    Now consider
    \begin{equation*}
        \gamma = (\id, \nabla |x_1|)_\#\tilde\rho = \gamma^0_0 = \gamma^0_1  \in \Pi(\rho_0, \mu_0).
    \end{equation*}
    By construction, $\spt \gamma \subset \graph(\partial \phi)$. The triangle inequality gives
\begin{equation*}
    \max_{i=0,1} W_q(\gamma, \gamma_i^\varepsilon) \geq \frac{1}{2}W_q(\gamma_0^\varepsilon, \gamma_1^\varepsilon) \geq 2^{-4}\varepsilon^{\theta/q}.
\end{equation*}
Together with the above bounds on $W_r(\tilde\rho,\rho_\varepsilon)$, the exponent in Theorem \ref{thrm: same subgrad plan control} is also sharp for any $q, r$ and $\theta$.
\end{proof}

\subsection{Sharpness for stability of maps and plans}

In our second example, we cannot use the crude support distance lower bound, and $W_q(\gamma_0, \gamma_1)$ is instead lower bounded by a form of measure theoretic pigeonhole principle. Roughly speaking, we show that although for every point in $\spt \gamma_0$ there is a nearby point in $\spt \gamma_1$, the mass given locally to certain regions of the product space by each plan is different, and so any coupling between the two plans must couple a certain amount of mass across a long distance.

The example shows that even for the fixed source estimate, the plan stability exponent is optimal. This automatically implies the same upper bound on the exponent for the distance between maps, since $W_q(\gamma_{\rho \to \mu_0}, \gamma_{\rho \to \mu_1}) \leq \|T_{\rho \to \mu_0} - T_{\rho \to \mu_1}\|_{L^q(\rho)}$. Thus although we do not explicitly calculate the map distance, this example also gives sharpness of the map stability, Theorem \ref{thrm: bi marginal map stability}.
\begin{proposition}
    \label{prop: sharp exponent stability}
    Let $d\geq 2$ and $\theta \in (0, 1]$. Let $\X$ and  $\Y$ be compact convex subsets of $\R^d$ with nonempty interior. There exists a probability measure $\rho$ on $\X$ satisfying \ref{eq: assumption eqn} for some $M>0$, and measures $\mu_0, (\mu_i)_{i\geq 1} \in \PP(\Y)$ such that for any $q \geq 1$,  $r \in [1, \infty]$ and
    \begin{equation*}
        \alpha > \begin{cases}
        \frac{r\theta}{q(r+\theta)} & r< \infty,\\
        \frac{\theta}{q} & r = \infty,
        \end{cases}
    \end{equation*}
    denoting by $\gamma_i \in \Gamma(\rho, \mu_i)$ the unique optimal transport plan,
    \begin{equation}
        \label{eq: sup the stability fails}
        \limsup_{i \to \infty} \frac{W_q(\gamma_{\rho \to \mu_0}, \gamma_{\rho \to \mu_i})}{W_r(\mu_0, \mu_i)^\alpha} = \infty. 
    \end{equation}
    Consequently the exponents in Theorems~\ref{thrm: bi marginal map stability} and \ref{thrm: bi marginal plan stability} are sharp.
\end{proposition}
\begin{proof}
    It suffices to construct one such compactly supported example in $\R^d$; then since the quadratic optimal transport problem is invariant with respect to dilations and translations of source and target, and $\X, \Y$ have non-empty interior, the same example can be dilated and translated inside these sets.  We use a modified form of the example given in \cite[Theorem 2]{letrouit2026unstable}. We construct $\rho$ out of multiple copies of cells of the form
    \begin{equation*}
        Q(l, h) := \left\{ z \in \R^d: 0 < z_1 < l, |z_2| <h/2, |z_k| < 1/2\; (3 \leq k \leq d)\right\}
    \end{equation*}
     for some $l, h>0$. (If $d=2$ we omit the last condition.) For a centre $c_i \in \R$ and offset $a_i>0$, a horizontal length $l_i>0$ and vertical thickness $h_i>0$, define
     \begin{align*}
        R_i^+ =& (c_i + a_i, 0,...,0) + Q(l_i, h_i),\\
        R_i^- =& (c_i - a_i, 0,...,0) - Q(l_i, h_i)
     \end{align*}
     in the sense of Minkowski. We will choose $a_i$ and $c_i$ at scales such that all cells are separated. Now for a collection of positive lengths $(l_i, h_i, a_i)_{i \in \mathbb{N}}$ and centres $(c_i)_{i \in \mathbb{N}}$ such that $\sum_{i \geq 1} l_i h_i^\theta < +\infty$ to be fixed later, we denote
     \begin{equation*}
        U := \bigcup_{i \geq 1} (R_i^+ \cup R_i^-), \quad \rho(x) := Z_\theta^{-1}|x_2|^{\theta - 1} \chi_U(x),
     \end{equation*}
     where 
     \begin{equation*}
        Z_\theta = \frac{2^{2-\theta}}{\theta} \sum_{i \geq 1} l_i h_i^\theta
     \end{equation*}
     is a normalising constant so that $\rho$ is a probability density. For any ball $B_\eta(x)$,
     \begin{align*}
        \rho(B_\eta(x)) \leq & \int_{x_2 - \eta}^{x_2 + \eta} \int_{B_\eta(x_{-2})} \rho(x) \di x_1 \di x_3 ... \di x_d \di x_2 \\
        \leq & \eta^{d-1}\beta_{d-1} Z_\theta^{-1} \int_{x_2 - \eta}^{x_2 +\eta} |x_2|^{\theta-1} \di x_2 \leq \frac{2}{\theta Z_\theta} \beta_{d-1} \eta^{d-1 + \theta},
     \end{align*}
     where we used that $\int_{x_2 - \eta}^{x_2 +\eta} |x_2|^{\theta-1} \di x_2$ is maximised at $x_2 = 0$. Hence $\rho$ satisfies \ref{eq: assumption eqn} with $M = 2\theta^{-1} Z_\theta^{-1} \beta_{d-1}$. Define
     \begin{equation*}
        \sigma_i := \rho(R_i^+) = \rho(R_i^-) = c_{\rho, \theta} l_i h_i^\theta, \quad \text{ where } c_{\rho, \theta} = \frac{2^{1-\theta}}{\theta Z_\theta}.
     \end{equation*}
     We introduce target points for the mass in the $i$th cells:
     \begin{equation*}
        y_i^+ = (c_i, a_i, 0,...,0), \quad\quad y_i^- = (c_i, -a_i, 0,...,0),
     \end{equation*}
     and their perturbations by $h_i$ in the $x_1$ coordinate:
     \begin{equation*}
        \hat y_i^+ = (c_i +h_i, a_i, 0,...,0), \quad\quad \hat y_i^- = (c_i -h_i, -a_i, 0,...,0).
     \end{equation*}
     We claim that every point $x^+ \in R_i^+$ is closer to $\hat y_i^+$ than $\hat y_i^-$, and similarly every $x^- \in R_i^-$ is closer to $\hat y_i^-$ than $\hat y_i^+$. Indeed, write $x^+ = (c_i + a_i, 0,...,0) + z_i^+$ for some $z_i^+ \in Q(l_i, h_i)$. The worst case is given by $z_i^+ = (0, -h_i/2,...)$ where the remaining coordinates contribute the common term $\sum_{k=3}^d|x_k^+|^2$ to both squared distances. In this limiting case,
     \begin{align*}
        \|x^+ - \hat y_i^+\|^2 =& (h_i- a_i)^2 + (a_i + h_i/2)^2 + \sum_{k=3}^d|x_k^+|^2,\\
        \|x^+ - \hat y_i^-\|^2 =& (h_i + a_i)^2 + (a_i - h_i/2)^2 + \sum_{k=3}^d|x_k^+|^2,
     \end{align*}
     so that
     \begin{equation*}
        \|x^+ - \hat y_i^-\|^2 - \|x^+ - \hat y_i^+\|^2 = 2a_i h_i >0. 
     \end{equation*}
     We define a reference target measure, as well as the $i$th perturbed target by 
     \begin{align*}
        \mu_0 =& \sum_{j \geq 1} \sigma_j (\delta_{y_j^+} + \delta_{y_j^-})\\
        \mu_i =& \mu_0 + \sigma_i(\delta_{\hat y_i^+} + \delta_{\hat y_i^-} -\delta_{y_i^+} - \delta_{y_i^-}).
     \end{align*}
     The perturbation $\mu_0 \to \mu_i$ simply moves the centred masses near the cells $R_i^+, R_i^-$ to their perturbed positions. We impose that
     \begin{align}
        \forall i \geq 2, \quad \min(c_i - c_{i-1}, c_{i+1} - c_i) \geq& 100 \max(l_i, h_i, a_i),\label{eq: constraint 1 on sequences}\\
         \forall i \geq 1, \quad a_i \geq& 100 h_i. \label{eq: constraint 2 on sequences}
     \end{align}
     The first ensures that clusters with different indices are well separated: the first two coordinate projections of the cells and targets indexed by $i$ lie much closer to each other than to those of any other cluster. The second guarantees that
     \begin{equation*}
        \|y_i^+ - \hat y_i^+\| \leq \|y_i^+ - \hat y_i^-\| \quad \text{ and } \quad \|y_i^- - \hat y_i^-\| \leq \|y_i^- - \hat y_i^+\|.
     \end{equation*}
     Consequently, by coupling  mass via $y_i^+ \leftrightarrow \hat y_i^+$ and $y_i^- \leftrightarrow \hat y_i^-$,
     \begin{equation*}
        W_r(\mu_0, \mu_i) \leq 2^{1/r} h_i \sigma_i^{1/r} \quad (r<+\infty), \qquad W_\infty(\mu_0,\mu_i)\leq h_i.
     \end{equation*}
     Since $\rho$ is absolutely continuous, the optimal plans from $\rho$ to $\mu_i$ are induced by maps. The optimal map from $\rho$ to $\mu_0$ is
     \begin{equation*}
        T_0(x) = \begin{cases}
            y_j^+ & x \in R_j^+ \cup R_j^- \text{ and } x_2 \geq 0,\\
            y_j^- & x \in R_j^+ \cup R_j^- \text{ and } x_2 < 0,
        \end{cases}
     \end{equation*}
     for each $j\geq1$, since every point is sent to the nearest point in $\spt \mu_0$. The optimal map from $\rho$ to $\mu_i$ is
     \begin{equation*}
        T_i(x) = \begin{cases}
            T_0(x) & x \notin R_i^+ \cup R_i^-,\\
            \hat y_i^+ & x \in R_i^+,\\
            \hat y_i^- & x \in R_i^-,\\
        \end{cases}
     \end{equation*}
     as again every point is sent to the nearest point in $\spt \mu_i$ by the above calculations.
     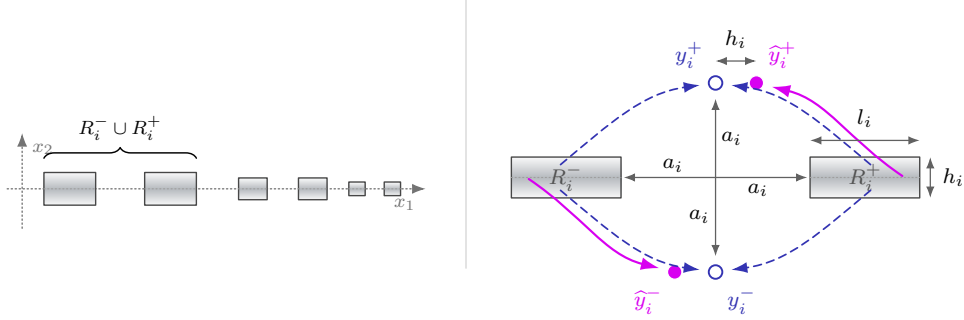
\begin{figure}[ht]
        \centering
        \begin{tikzpicture}[
        x=1cm,
        y=1cm,
        >=Latex,
        line cap=round,
        line join=round,
        every node/.style={font=\small},
        ref flow/.style={-{Latex[length=2.15mm,width=1.45mm]},draw=targetblue,densely dashed,line width=.65pt},
        pert flow/.style={-{Latex[length=2.45mm,width=1.7mm]},draw=sourcecolour,line width=.82pt},
        reference target/.style={circle,draw=targetblue,fill=white,line width=.72pt,minimum size=5.1pt,inner sep=0pt},
        perturbed target/.style={circle,draw=white,fill=sourcecolour,line width=.45pt,minimum size=5.4pt,inner sep=0pt},
        dimension/.style={<->,draw=black!62,line width=.48pt,{Latex[length=1.5mm]}-{Latex[length=1.5mm]}}
        ]

        \path[use as bounding box] (0,-1.62) rectangle (13.2,2.55);

        \newcommand{\densityblock}[3]{%
        \shade[bottom color=cellgray!45,top color=cellgray!4]
            (#1,0) rectangle (#2,#3);
        \shade[top color=cellgray!45,bottom color=cellgray!4]
            (#1,0) rectangle (#2,-#3);
        \draw[black!72,line width=.52pt] (#1,-#3) rectangle (#2,#3);
        \draw[black!38,densely dotted,line width=.42pt] (#1,0)--(#2,0);
        }

        \begin{scope}[shift={(3.30,0)}]
        \begin{scope}[scale=.90,transform shape]
        \begin{scope}[shift={(-3.355,0)}]

        \draw[black!55,densely dotted,line width=.5pt,-{Latex[length=1.8mm]}]
            (.28,0)--(6.43,0) node[below left=1pt,font=\footnotesize] {$x_1$};
        \draw[black!55,densely dotted,line width=.5pt,-{Latex[length=1.8mm]}]
            (.50,-.58)--(.50,.82) node[below right=1pt,font=\footnotesize] {$x_2$};

        \densityblock{.82}{1.58}{.24}
        \densityblock{2.30}{3.06}{.24}
        \draw[decorate,decoration={brace,amplitude=4pt},line width=.55pt]
            (.82,.45)--(3.06,.45)
            node[midway,above=4pt,font=\footnotesize] {$R_i^-\cup R_i^+$};

        \densityblock{3.68}{4.10}{.16}
        \densityblock{4.56}{4.98}{.16}
        \densityblock{5.30}{5.54}{.10}
        \densityblock{5.82}{6.06}{.10}
        \end{scope}
        \end{scope}
        \end{scope}

        \draw[black!18,line width=.45pt] (6.60,-1.05)--(6.60,2.55);

        \begin{scope}[shift={(9.90,.15)}]

        \densityblock{-2.70}{-1.25}{.27}
        \densityblock{ 1.25}{ 2.70}{.27}
        \node[font=\footnotesize,black!68,fill=white,fill opacity=0,text opacity=1,inner sep=1.2pt]
            at (-1.98,0) {$R_i^-$};
        \node[font=\footnotesize,black!68,fill=white,fill opacity=0,text opacity=1,inner sep=1.2pt]
            at ( 1.98,0) {$R_i^+$};

        \node[reference target] (yp) at (0,1.25) {};
        \node[perturbed target] (hyp) at (.54,1.25) {};
        \node[reference target] (ym) at (0,-1.25) {};
        \node[perturbed target] (hym) at (-.54,-1.25) {};
        \node[font=\footnotesize,above left=2pt and 1pt, text=targetblue] at (yp) {$y_i^+$};
        \node[font=\footnotesize,above right=2pt and 1pt,text=sourcecolour] at (hyp) {$\widehat y_i^+$};
        \node[font=\footnotesize,below right=2pt and 1pt, text=targetblue] at (ym) {$y_i^-$};
        \node[font=\footnotesize,below left=2pt and 1pt,text=sourcecolour] at (hym) {$\widehat y_i^-$};

        \draw[ref flow,shorten >=3pt] (-2.05,.17) to[out=40,in=188] (yp);
        \draw[ref flow,shorten >=3pt] ( 2.05,.17) to[out=140,in=-8] (yp);
        \draw[ref flow,shorten >=3pt] (-2.05,-.17) to[out=-40,in=172] (ym);
        \draw[ref flow,shorten >=3pt] ( 2.05,-.17) to[out=-140,in=8] (ym);

        \draw[pert flow,shorten >=3pt] (-2.47,-.02) to[out=-34,in=166] (hym);
        \draw[pert flow,shorten >=3pt] ( 2.47,.02) to[out=146,in=-14] (hyp);

        \draw[draw=black!62,line width=.48pt,-{Latex[length=1.5mm]},shorten >=1pt]
            (0,0)--(-1.25,0)
            node[pos=.45,anchor=south,inner sep=2pt,yshift=0pt,font=\footnotesize] {$a_i$};
        \draw[draw=black!62,line width=.48pt,-{Latex[length=1.5mm]},shorten >=1pt]
            (0,0)--(1.25,0)
            node[pos=.45,anchor=north,inner sep=2pt,yshift=0pt,font=\footnotesize] {$a_i$};
        \draw[draw=black!62,line width=.48pt,-{Latex[length=1.5mm]},shorten >=3pt]
            (0,0)--(yp)
            node[pos=.45,anchor=west,inner sep=2pt,xshift=0pt,font=\footnotesize] {$a_i$};
        \draw[draw=black!62,line width=.48pt,-{Latex[length=1.5mm]},shorten >=3pt]
            (0,0)--(ym)
            node[pos=.45,anchor=east,inner sep=2pt,xshift=0pt,font=\footnotesize] {$a_i$};
        \draw[dimension] (1.25,.49)--(2.70,.49)
            node[midway,above=1.4pt,font=\footnotesize] {$l_i$};
        \draw[dimension] (2.84,-.27)--(2.84,.27)
            node[midway,right=1.1pt,font=\footnotesize] {$h_i$};
        \draw[dimension] (0,1.54)--(.54,1.54)
            node[midway,above=1.4pt,font=\footnotesize] {$h_i$};
        \end{scope}

        \end{tikzpicture}
        \caption{Left: the support of $\rho$. Right: a magnified view of the $i$th pair of cells, with the optimal maps \textcolor{targetblue}{$T_0$: $\rho \to \mu_0$ (blue, dashed)} and \textcolor{sourcecolour}{$T_i: \rho \to\mu_i$ (pink, solid)}.}
        \label{fig: stability example}
     \end{figure}
     
    We now prove a lower bound on the $W_q$ distance between the transport plans $\gamma_i = (\id, T_i)_\#\rho$. Unlike in the sharp quantitative uniqueness example, the supports are not cleanly separated. Set
    \begin{equation*}
        G_i = (R_i^+ \times \{ \hat y_i^+\}) \cup (R_i^- \times \{ \hat y_i^-\}), \quad \gamma_i(G_i) = 2 \sigma_i,
    \end{equation*}
    and
    \begin{align*}
        H_i = \big((R_i^+ \cap \{ x_2 \geq 0\}) \times \{ y_i^+\}\big) \cup& \big((R_i^- \cap \{ x_2 \leq 0\}) \times \{y_i^-\}\big), \\
        \gamma_0(H_i) =& \sigma_i.
    \end{align*}
    Consider any coupling $\pi \in \Pi(\gamma_0, \gamma_i)$, with variables $(x_0,y_0,x_i,y_i)$. For $\pi$-almost every such point with $(x_i,y_i)\in G_i$, if
    \begin{equation*}
        \max(\|x_0 - x_i\|, \| y_0 - y_i\|) < 2 a_i,
    \end{equation*}
    we must have $(x_0, y_0) \in H_i$ by the separation of cells and $a_i \geq 100 h_i$. It follows that
    \begin{equation*}
        \pi( \max(\|x_0 - x_i\|, \|y_0 - y_i\|) \geq 2a_i) \geq \sigma_i
    \end{equation*}
    for any $\pi \in \Pi(\gamma_0, \gamma_i)$ -- this is effectively a measure theoretic pigeonhole principle, for couplings rather than bijections. It follows that
    \begin{equation*}
        W_q(\gamma_0, \gamma_i) \geq 2 a_i \sigma_i^{1/q}.
    \end{equation*}
    We now fix $(l_i, h_i, c_i, a_i)_{i \in \mathbb{N}}$, taking
    \begin{equation*}
        a_i = l_i = (i+i_0)^{-2}, \quad h_i = 2^{-(i+i_0)}, \quad c_1=0,\quad c_{i+1} - c_i = 100(a_i + a_{i+1}).
    \end{equation*}
    We fix $i_0$ large enough that both conditions \eqref{eq: constraint 1 on sequences} and \eqref{eq: constraint 2 on sequences} hold for all $i\geq1$. All cells, sums and the normalising constant $Z_\theta$ use these sequences. Since
    \begin{equation*}
        \sum_{i \geq 1} a_i < \infty \quad \text{ and } \quad \sum_{i \geq1} l_i h_i^\theta < \infty,
    \end{equation*}
    $\rho$ is a compactly supported probability measure, with $\spt \rho = \overline U$. Assume that for some constant $C$ independent of $i$, and some $\alpha >0$,
    \begin{equation*}
        W_q(\gamma_0, \gamma_i) \leq C W_r(\mu_0, \mu_i)^\alpha. 
    \end{equation*}
    Then, up to constants independent of $i$,
    \begin{equation*}
        W_r(\mu_0, \mu_i)^\alpha \leq 2 c_{\rho, \theta, r} \big(l_i^{1/r}h_i^{(r+\theta)/r}\big)^\alpha 
    \end{equation*}
    and
    \begin{equation*}
        W_q(\gamma_0, \gamma_i) \geq 2 c_{\rho, \theta, q} l_i^{(q+1)/q} h_i^{\theta/q}
    \end{equation*}
    so that for all $i$,
    \begin{equation*}
        l_i^{\left(\frac{q+1}{q} - \frac{\alpha}{r}\right)} h_i^{\left(\frac{\theta}{q} - \frac{\alpha(r+\theta)}{r}\right)} \leq C_{\rho, \theta, r, q}.
    \end{equation*}
    Since $h_i$ decays exponentially fast compared to the polynomial $l_i$, the only way this is possible is if the exponent is nonnegative, i.e.
    \begin{equation*}
        \frac{\theta}{q} - \frac{\alpha(r+\theta)}{r} \geq 0 \implies \alpha \leq \frac{r\theta}{q(r+\theta)}.
    \end{equation*}
    When $r=\infty$, the bound $W_\infty(\mu_0,\mu_i)\leq h_i$ gives
    \begin{equation*}
        \frac{W_q(\gamma_0,\gamma_i)}{W_\infty(\mu_0,\mu_i)^\alpha}
        \geq 2c_{\rho,\theta}^{1/q}a_i l_i^{1/q}h_i^{\theta/q-\alpha}\longrightarrow\infty
        \quad\text{if }\alpha>\theta/q.
    \end{equation*}
    Consequently, conclusions of Theorems \ref{thrm: bi marginal map stability} and \ref{thrm: bi marginal plan stability} cannot hold with any exponent larger than the one proven.
\end{proof}

\subsection{Sharpness for coercivity}

The above example can be adapted to show that the coercivity theorems, Theorems \ref{thrm: coercivity general marginals} and \ref{thrm: coercivity rough measures}, are also sharp in exponent. The general form of this is already sharp as a consequence of the sharpness for stability with $r=1$. The example below uses fixed marginals, establishing that the exponent in terms of the suboptimality deficit is also sharp.

\begin{proposition}
    \label{prop: sharp exponent coercivity}
    Let $d\geq 2$ and $\theta \in (0, 1]$. Let $\X$ and $\Y$ be compact convex sets in $\R^d$ with nonempty interior. There exists a probability measure $\rho$ on $\X$ satisfying \ref{eq: assumption eqn} for some $M>0$, a measure $\mu \in \PP(\Y)$, and transport plans $(\gamma_i)_{i \in \mathbb{N}} \in \Pi(\rho, \mu)$ such that for any $q \in [1, +\infty)$ and
    \begin{equation*}
    \alpha > \frac{\theta}{q(1+\theta)},
    \end{equation*}
    it holds that
    \begin{equation}
    \label{eq: sup the coercivity fails}
    \limsup_{i \to \infty} \frac{W_q(\gamma_{\rho \to \mu}, \gamma_i)}{\left(\int \langle x, y \rangle \di (\gamma_{\rho \to \mu} - \gamma_i)\right)^\alpha} = \infty. 
    \end{equation}
    Consequently the exponents in Theorems~\ref{thrm: coercivity general marginals} and \ref{thrm: coercivity rough measures} are sharp.
\end{proposition}
\begin{proof}
    We use the same construction as for the stability, taking the same $\rho$, and $\mu= \mu_0$. The optimal plan is induced by the map $T_0$ already calculated. For a candidate suboptimal plan, we choose $\gamma_i \in \Pi(\rho, \mu)$ induced by the (non-optimal) map
    \begin{equation*}
        T_i(x) = \begin{cases}
            T_0(x) & x \notin R_i^+ \cup R_i^-\\
            y_i^+ & x \in R_i^+\\
            y_i^- & x \in R_i^-.
        \end{cases}
    \end{equation*}
    Using the same pigeonhole argument lower bound as for the stability, we deduce
    \begin{equation*}
        W_q(\gamma_{\rho \to \mu}, \gamma_i) = 2 a_i \sigma_i^{1/q},
    \end{equation*}
    where the equality is a consequence of the upper bound by $\|T_0 - T_i\|_{L^q(\rho)} = 2 a_i \sigma_i^{1/q}$. To compute the excess transport cost, we compute 
    \begin{equation*}
        \Delta_i := \int \|T_i(x) - x\|^2 - \|T_0(x) - x\|^2 \di \rho(x) = \int \|x - y\|^2 \di (\gamma_i - \gamma_{\rho \to \mu}) = 2 \int \langle x, y \rangle \di (\gamma_{\rho \to \mu} - \gamma_i),
    \end{equation*}
    since the marginals are equal. The maps are equal on $R_i^+ \cap \{ x_2 \geq 0\}$ and $R_i^- \cap \{ x_2 \leq 0\}$, and for $x \in R_i^+ \cap \{ x_2 \leq 0\}$ we have
    \begin{align*}
        \|T_i(x) - x\|^2 - \|T_0(x) - x\|^2 =& \|x - y_i^+\|^2 - \|x - y_i^-\|^2\\
        =& (a_i - x_2)^2 - (a_i + x_2)^2 = 4 a_i |x_2|
    \end{align*}
    with the value on $R_i^- \cap \{ x_2 \geq 0\}$ equal by symmetry. Consequently,
    \begin{equation*}
        \Delta_i = 2a_i\int_{R_i^+ \cup R_i^-} |x_2| \di \rho(x) = \frac{4 a_i l_i}{Z_\theta} \int_{-h_i/2}^{h_i/2} |x_2|^\theta \di x_2 = \frac{4}{Z_\theta(1+ \theta)2^\theta} a_i l_i h_i^{1+\theta}.
    \end{equation*}
    Finally, for $\alpha >0$ and the same sequences as the previous proposition,
    \begin{equation*}
        \frac{W_q(\gamma_{\rho \to \mu}, \gamma_i)}{\Delta_i^\alpha} = C_{\rho, \theta, q, \alpha} a_i^{\left(\frac{q+1}{q} - 2 \alpha\right)} h_i^{\left(\frac{\theta}{q} - \alpha(1+\theta) \right)}
    \end{equation*}
    so that the sequence only remains bounded as $i \to \infty$ when
    \begin{equation*}
        \frac{\theta}{q} - \alpha(1+\theta) \geq 0 \implies \alpha \leq \frac{\theta}{q(1+\theta)}.
    \end{equation*}
    Thus Theorems \ref{thrm: coercivity general marginals} and \ref{thrm: coercivity rough measures} are also sharp in exponent when comparing plans with fixed marginals as required.
\end{proof}

\appendix

\section{Further estimates and an alternative proof of stability}

In this appendix, we present a second proof of Theorem~\ref{thrm: bi marginal plan stability}, which uses the same subgradient plan comparison theorem, Theorem \ref{thrm: same subgrad plan control}.  We also prove a strong convexity result for the Kantorovich functional, generalising the result of \cite{merigot2026sharp} to allow for our exponents $q$ and $\theta$.

\subsection{An alternative proof of Theorem~\ref{thrm: bi marginal plan stability}: bi-marginal plan stability}
We combine the same subgradient control with the bi-marginal map stability Theorem~\ref{thrm: bi marginal map stability}. This recovers the exponent of Theorem~\ref{thrm: bi marginal plan stability}, with a possibly different constant.
Let $\rho_0$ satisfy \ref{eq: assumption eqn} and let $\phi_i$ be Brenier potentials for the $\rho_i\to\mu_i$ transport problems, chosen as in \eqref{eq: double legendre trans eq}. Set $\hat\gamma=(\id,\nabla\phi_1)_\#\rho_0$. Then by the same subgradient control Theorem~\ref{thrm: same subgrad plan control}, with $r=1$,
\begin{equation*}
    W_q(\gamma_{\rho_1 \to \mu_1}, \hat \gamma) \leq C W_1(\rho_0, \rho_1)^{\frac{\theta}{q(1+\theta)}}
\end{equation*}
for some $C(\X,\Y,\theta,q,M)>0$. By coupling the two graph plans through their common source and applying Theorem~\ref{thrm: bi marginal map stability},
\begin{align*}
    W_q(\gamma_{\rho_0 \to \mu_0}, \hat \gamma) &\leq \|\nabla\phi_0-\nabla\phi_1\|_{L^q(\rho_0)}\\
    &\leq C \Big(W_1(\rho_0, \rho_1) + W_1(\mu_0, \mu_1)\Big)^\frac{\theta}{q(1+\theta)}
\end{align*}
for some $C(\X,\Y,\theta,q,M)>0$. Bi-marginal plan stability then follows by the triangle inequality.

\subsection{Strong convexity of the Kantorovich functional via the Jensen divergence}

We study the strong convexity properties of the Kantorovich functional, defined for compact sets $\X, \Y$ as $\K_\rho : \mathcal{C}^0(\Y) \to \R$,
 \begin{equation*}
    \K_\rho(\psi) = \int_{\X} \psi^*(x) \di \rho(x).
\end{equation*}
$\K_\rho$ is convex and 1-Lipschitz in the uniform norm, and when $\rho$ satisfies \ref{eq: assumption eqn}, $\K_\rho$ is everywhere Gâteaux differentiable, with
\begin{equation*}
    \nabla \K_\rho(\psi) = -\nabla\psi^*_\#\rho.
\end{equation*}
For more details see \cite{letrouit2025lectures}. We measure the strong convexity of $\K_\rho$ using the Jensen divergence $J_\rho : \mathcal{C}^0(\Y) \times \mathcal{C}^0(\Y)\to \R$,
\begin{equation*}
    J_\rho(\psi_0, \psi_1) := \frac{1}{2}\Big(\K_\rho(\psi_0) + \K_\rho(\psi_1)\Big) - \K_\rho\Big(\frac{1}{2}(\psi_0 + \psi_1)\Big).
\end{equation*}
The following theorem generalises \cite[Theorem 2.2]{merigot2026sharp}, which corresponds to our result for $\theta = 1$ and $q = 2$.

\begin{samepage}
\begin{theorem}
\label{thrm: strong convexity with general exponents}
    Let $q \in [1, +\infty)$, let $\X, \Y \subset \R^d$ be compact and convex for $d\geq2$, and let $\rho$ be a probability measure on $\X$ satisfying \ref{eq: assumption eqn}. Then for all $\psi_0, \psi_1 \in \C^0(\Y)$,
    \begin{equation}
    \label{eq: kanto strong conv Lq}
        \|\nabla \psi_0^* - \nabla\psi^*_1\|_{L^q(\rho)}^{\frac{q(1 + \theta)}{\theta}}\leq C^{-1}J_\rho(\psi_0, \psi_1) 
    \end{equation}
    where
    \begin{equation*}
        C = \theta(1+\theta)^{-\frac{1+\theta}{\theta}} 4^{-\frac{1+\theta}{\theta}} \left(k_q D_\Y^{q+\theta(q-1)}c_d M D_\X^{d-1}\right)^{-\frac{1}{\theta}}
    \end{equation*}
    and $k_q$ is defined in the proof.
\end{theorem}
\end{samepage}
\begin{proof}
    For $i=0,1$, $\phi_i = \psi_i^*$ is differentiable $\rho$-almost everywhere. Set $u(x) = \nabla \phi_1(x) - \nabla \phi_0(x)$ and $A = \|u\|_{L^q(\rho)}$. If $A=0$, the result follows from convexity of $\K_\rho$, so assume $A>0$. By definition,
    \begin{equation}
    \label{eq: inf convolution rep for jensen}
        J_\rho(\psi_0, \psi_1) = \int_\X j(x) \di \rho(x)
    \end{equation}
    where
    \begin{equation*}
        j(x) = \frac{\phi_0(x)+\phi_1(x)}{2} - \left(\frac{\psi_0+\psi_1}{2}\right)^*(x).
    \end{equation*}
    For every $h\in\R^d$, the definition of the conjugate gives
    \begin{equation*}
        \left(\frac{\psi_0+\psi_1}{2}\right)^*(x) \leq \frac{\phi_0(x+h)+\phi_1(x-h)}{2}.
    \end{equation*}
    Fix a point $x$ where both $\phi_i$ are differentiable. For some $h\in\R^d$ to be fixed later, choose $y_h^0\in\partial\phi_0(x+h)$ and $y_h^1\in\partial\phi_1(x-h)$. Subgradient inequalities give
    \begin{equation*}
        \phi_0(x) - \phi_0(x+h) \geq \langle y_h^0, -h \rangle \quad \text{and}\quad \phi_1(x) - \phi_1(x-h) \geq \langle y_h^1, h \rangle
    \end{equation*}
    so that
    \begin{align*}
        j(x) \geq& \frac{1}{2}\langle y_h^1 - y_h^0, h\rangle\\
        =& \frac{1}{2}\langle u(x), h\rangle + \frac{1}{2}\langle y_h^1 - \nabla \phi_1(x), h \rangle + \frac{1}{2}\langle \nabla\phi_0(x) - y_h^0, h \rangle.
    \end{align*}
    We choose, for some $\lambda>0$ to be fixed later,
    \begin{equation*}
        h_\lambda(x) = \begin{cases}
            \lambda\|u(x) \|^{q-2} u(x) & u(x) \neq 0, \\
            0& u(x) = 0.
        \end{cases}
    \end{equation*}
    Set $\eta_\lambda := D_\Y^{q-1}\lambda$, so that $\|h_\lambda(x)\|\leq\eta_\lambda$. Set $k_1 = 1$ and for $q>1$,
    \begin{equation*}
        k_q = \frac{4^{q-1}(q-1)^{q-1}}{q^q} \geq \frac{3}{4}.
    \end{equation*}
    For $q>1$, Young's inequality $a^{q-1}b \leq a^q/4 + k_q b^q$ implies
    \begin{equation}
        \label{eq: pointwise j rho lower bound}
        j(x) \geq \frac{\lambda}{4}\|\nabla \phi_1(x) - \nabla\phi_0(x) \|^{q} - \frac{k_q\lambda}{2} \left(\big(D_{\eta_{\lambda}} \phi_0(x)\big)^{q} + \big(D_{\eta_\lambda} \phi_1(x)\big)^{q} \right),
    \end{equation}
    and when $q=1$ the same inequality holds directly since $\|h_\lambda(x)\| \in \{0, \lambda\}$. For any $\lambda\geq 0$, Theorem~\ref{thrm: quant reg} gives
    \begin{align*}
         \int_{\X} k_q \big(D_{\eta_{\lambda}} \phi_i(x)\big)^{q} \di \rho(x) \leq& k_q D_\Y^q c_d M D_\X^{d-1} \eta_\lambda^\theta\\
        =& k_q D_\Y^{q+\theta(q-1)}c_d M D_\X^{d-1}\lambda^\theta = \overline C \lambda^\theta.
    \end{align*}
    Thus integrating \eqref{eq: pointwise j rho lower bound} with respect to $\rho$,
    \begin{equation}
        \label{eq: J rho lambda lower bound}
        J_{\rho}(\psi_0, \psi_1) \geq \frac{\lambda}{4}A^q - \overline C\lambda^{1+\theta}.
    \end{equation}
    We choose
    \begin{equation*}
        \lambda^\theta = \frac{A^q}{4(1+\theta)\overline C},
    \end{equation*}
    so that substituting the chosen value of $\lambda$ into \eqref{eq: J rho lambda lower bound} gives
    \begin{equation*}
        J_\rho(\psi_0, \psi_1) \geq  CA^{\frac{q(1+ \theta)}{\theta}}
    \end{equation*}
    for the $C$ given in the statement of the theorem, so the proof is complete.
\end{proof}

\par\bigskip

\noindent \textbf{Funding:}
This research benefited from the support of the FMJH Program Gaspard Monge for optimization and operations research and their interactions with data science.

\noindent \textbf{Statement of generative AI use:} The author acknowledges the use of GPT 5.6 Sol during the research and writing of this article, as well as GPT 6 Astra in the final revision phases. In particular, Lemma~\ref{lem: Fenchel bound distance} and Theorem~\ref{thrm: same subgrad plan control} were found by the AI agent, the first of which was used to give a more direct proof of the stability results rather than the author's initial approach passing via the Jensen divergence based on \cite{merigot2026sharp}, which is given in the appendix. The sharpness examples were constructed by the AI agent, which also wrote the TikZ code for the graphics, whilst the body of the paper was drafted by the author. Subsequent AI assistance was used for stylistic and language corrections. The author takes sole responsibility for the veracity of the mathematical content of the paper.

\noindent \textbf{Conflicts of interest:} The author declares no conflicts of interest.

\noindent \textbf{Data availability:} No datasets were generated or analysed in this study.

\printbibliography

@article{brenier1991polar,
  title={Polar factorization and monotone rearrangement of vector-valued functions},
  author={Brenier, Yann},
  journal={Communications on Pure and Applied Mathematics},
  volume={44},
  number={4},
  pages={375--417},
  year={1991},
  publisher={Wiley Online Library}
}

@article{delalande2023quantitative,
  title={Quantitative stability of optimal transport maps under variations of the target measure},
  author={Delalande, Alex and M{\'e}rigot, Quentin},
  journal={Duke Mathematical Journal},
  volume={172},
  number={17},
  pages={3321--3357},
  year={2023},
  publisher={Duke University Press}
}

@article{carlier2024barycentres,
  title={Quantitative stability of barycenters in the {W}asserstein space},
  author={Carlier, Guillaume and Delalande, Alex and M{\'e}rigot, Quentin},
  journal={Probability Theory and Related Fields},
  volume={188},
  number={3--4},
  pages={1257--1286},
  year={2024},
  publisher={Springer}
}

@article{letrouit2024gluing,
  title={Gluing methods for quantitative stability of optimal transport maps},
  author={Letrouit, Cyril and M{\'e}rigot, Quentin},
  journal={To appear in Annales Scientifiques de l'Ecole Normale Supérieure},
  year={2024}
}

@article{carlier2024pushforward,
  title={Quantitative stability of the pushforward operation by an optimal transport map},
  author={Carlier, Guillaume and Delalande, Alex and M{\'e}rigot, Quentin},
  journal={Foundations of Computational Mathematics},
  volume={25},
  pages={1259--1286},
  year={2025},
  doi={10.1007/s10208-024-09669-4},
  publisher={Springer}
}

@book{letrouit2025lectures,
  title={Quantitative stability in optimal transport},
  author={Letrouit, Cyril},
  series={Cours spécialisés},
  publisher={Société Mathématique de France},
  pubstate={forthcoming},
  url={https://www.imo.universite-paris-saclay.fr/~cyril.letrouit/teaching/Peccotfinal.pdf}
}

@article{li2021quantitative,
  title={Quantitative stability and error estimates for optimal transport plans},
  author={Li, Wenbo and Nochetto, Ricardo H},
  journal={IMA Journal of Numerical Analysis},
  volume={41},
  number={3},
  pages={1941--1965},
  year={2021},
  publisher={Oxford University Press}
}

@article{gigli2011holder,
  title={On {H}{\"o}lder continuity-in-time of the optimal transport map towards measures along a curve},
  author={Gigli, Nicola},
  journal={Proceedings of the Edinburgh Mathematical Society},
  volume={54},
  number={2},
  pages={401--409},
  year={2011},
  publisher={Cambridge University Press}
}

@article{gallouet2018lagrangian,
  title={A Lagrangian scheme {\`a} la {B}renier for the incompressible {E}uler equations},
  author={Gallou{\"e}t, Thomas O and M{\'e}rigot, Quentin},
  journal={Foundations of Computational Mathematics},
  volume={18},
  number={4},
  pages={835--865},
  year={2018},
  publisher={Springer}
}

@article{berman2021convergence,
  title={Convergence rates for discretized {M}onge--{A}mp{\`e}re equations and quantitative stability of optimal transport},
  author={Berman, Robert J},
  journal={Foundations of Computational Mathematics},
  volume={21},
  number={4},
  pages={1099--1140},
  year={2021},
  publisher={Springer}
}

@article{ford2025quantitative,
  title={Quantitative Stability in Discrete Optimal Transport},
  author={Ford, William},
  journal={arXiv preprint arXiv:2510.17407},
  year={2025}
}

@misc{ford2026quantitative,
  title={Quantitative uniqueness of {Kantorovich} potentials},
  author={Ford, William},
  eprint={2603.29595},
  eprinttype={arxiv},
  year={2026}
}

@misc{merigot2026sharp,
  title={Sharp stability of {Brenier} maps via quantitative regularity of potentials},
  author={M{\'e}rigot, Quentin},
  year={2026},
  howpublished={Preprint, HAL hal-05616391},
  url={https://hal.science/hal-05616391v1}
}

@article{letrouit2026unstable,
  title={Unstable optimal transport maps},
  author={Letrouit, Cyril},
  journal={Comptes Rendus. Math{\'e}matique},
  volume={364},
  number={G2},
  pages={333--344},
  year={2026}
}

@book{ambrosio2000functions,
  title={Functions of bounded variation and free discontinuity problems},
  author={Ambrosio, Luigi and Fusco, Nicola and Pallara, Diego},
  year={2000},
  publisher={Oxford University Press}
}

@article{acosta2004optimal,
  title={An optimal {Poincar{\'e}} inequality in {$L^1$} for convex domains},
  author={Acosta, Gabriel and Dur{\'a}n, Ricardo},
  journal={Proceedings of the American Mathematical Society},
  volume={132},
  number={1},
  pages={195--202},
  year={2004}
}

@article{carlier2023fenchel,
  title={{Fenchel--Young} inequality with a remainder and applications to convex duality and optimal transport},
  author={Carlier, Guillaume},
  journal={SIAM Journal on Optimization},
  volume={33},
  number={3},
  pages={1463--1472},
  year={2023},
  publisher={SIAM}
}

@article{gallouet2025strong,
  title={Strong c-concavity and stability in optimal transport},
  author={Gallou{\"e}t, Anatole and M{\'e}rigot, Quentin and Thibert, Boris},
  journal={Journal de Math{\'e}matiques Pures et Appliqu{\'e}es},
  volume={205},
  eid={103773},
  year={2026},
  doi={10.1016/j.matpur.2025.103773},
  publisher={Elsevier}
}

@article{manole2024plugin,
  title={Plugin estimation of smooth optimal transport maps},
  author={Manole, Tudor and Balakrishnan, Sivaraman and Niles-Weed, Jonathan and Wasserman, Larry},
  journal={The Annals of Statistics},
  volume={52},
  number={3},
  pages={966--998},
  year={2024},
  publisher={Institute of Mathematical Statistics}
}

@article{cazelles2026statistical,
  title={Statistical estimation of {Monge} transport maps via {Brenier} potentials},
  author={Cazelles, Elsa and Pauwels, Edouard and Portales, L{\'e}o},
  journal={arXiv preprint arXiv:2604.22366},
  year={2026}
}

@article{mccann1995existence,
  title={Existence and uniqueness of monotone measure-preserving maps},
  author={McCann, Robert J},
  journal={Duke Mathematical Journal},
  volume={80},
  number={2},
  pages={309--323},
  year={1995},
  doi={10.1215/S0012-7094-95-08013-2}
}

@article{gigli2011inverse,
  title={On the inverse implication of {Brenier--McCann} theorems and the structure of {($\mathcal{P}_2(M), W_2$)}},
  author={Gigli, Nicola},
  journal={Methods Appl. Anal},
  volume={18},
  number={2},
  pages={127--158},
  year={2011}
}

@article{zajivcek1979differentiation,
  title={On the differentiation of convex functions in finite and infinite dimensional spaces},
  author={Zaj{\'\i}{\v{c}}ek, Lud{\v{e}}k},
  journal={Czechoslovak Mathematical Journal},
  volume={29},
  number={3},
  pages={340--348},
  year={1979},
  publisher={Institute of Mathematics, Academy of Sciences of the Czech Republic}
}

@article{gutierrez2022estimates,
  title={{$L^\infty$}-estimates in optimal transport for non quadratic costs},
  author={Guti{\'e}rrez, Cristian E and Montanari, Annamaria},
  journal={Calculus of Variations and Partial Differential Equations},
  volume={61},
  number={5},
  eid={163},
  year={2022},
  publisher={Springer}
}

@article{bouchitte2007new,
  title={A new {$L^\infty$} estimate in optimal mass transport},
  author={Bouchitt{\'e}, Guy and Jimenez, Chlo{\'e} and Rajesh, M},
  journal={Proceedings of the American Mathematical Society},
  volume={135},
  number={11},
  pages={3525--3535},
  year={2007},
  publisher={American Mathematical Society}
}

@article{el2012bayesian,
  title={Bayesian inference with optimal maps},
  author={El Moselhy, Tarek A and Marzouk, Youssef M},
  journal={Journal of Computational Physics},
  volume={231},
  number={23},
  pages={7815--7850},
  year={2012},
  publisher={Elsevier}
}

@article{alberti1992singularities,
  title={On the singularities of convex functions},
  author={Alberti, Giovanni and Ambrosio, Luigi and Cannarsa, Piermarco},
  journal={Manuscripta Math},
  volume={76},
  number={3-4},
  pages={421--435},
  year={1992}
}

@article{alberti1994structure,
  title={On the structure of singular sets of convex functions},
  author={Alberti, Giovanni},
  journal={Calculus of Variations and Partial Differential Equations},
  volume={2},
  number={1},
  pages={17--27},
  year={1994},
  publisher={Springer}
}

@article{courty2016optimal,
  title={Optimal transport for domain adaptation},
  author={Courty, Nicolas and Flamary, R{\'e}mi and Tuia, Devis and Rakotomamonjy, Alain},
  journal={IEEE transactions on pattern analysis and machine intelligence},
  volume={39},
  number={9},
  pages={1853--1865},
  year={2016},
  publisher={IEEE}
}

@inproceedings{makkuva2020optimal,
  title={Optimal transport mapping via input convex neural networks},
  author={Makkuva, Ashok and Taghvaei, Amirhossein and Oh, Sewoong and Lee, Jason},
  booktitle={International Conference on Machine Learning},
  pages={6672--6681},
  year={2020},
  organization={PMLR}
}

@inproceedings{arjovsky2017wasserstein,
  title={Wasserstein generative adversarial networks},
  author={Arjovsky, Martin and Chintala, Soumith and Bottou, L{\'e}on},
  booktitle={International conference on machine learning},
  pages={214--223},
  year={2017},
  organization={Pmlr}
}

@article{carlier2016vector,
  title={Vector quantile regression: an optimal transport approach},
  author={Carlier, Guillaume and Chernozhukov, Victor and Galichon, Alfred},
  year={2016}
}

@article{deb2023multivariate,
  title={Multivariate rank-based distribution-free nonparametric testing using measure transportation},
  author={Deb, Nabarun and Sen, Bodhisattva},
  journal={Journal of the American Statistical Association},
  volume={118},
  number={541},
  pages={192--207},
  year={2023},
  publisher={Taylor \& Francis}
}

@article{budd2015geometry,
  title={The geometry of r-adaptive meshes generated using optimal transport methods},
  author={Budd, CJ and Russell, Robert D and Walsh, E},
  journal={Journal of Computational Physics},
  volume={282},
  pages={113--137},
  year={2015},
  publisher={Elsevier}
}


\end{document}